\documentclass[a4paper]{amsart}
\usepackage{amsmath,amsthm,amssymb,mathtools,latexsym,epic,bbm,comment}
\usepackage{graphicx,enumerate,stmaryrd,tabularx}
\usepackage[all,2cell]{xy}
\xyoption{2cell}

\font\sc=rsfs10
\newcommand{\cF}{\sc\mbox{F}\hspace{1.0pt}}

\usepackage{todonotes}

\newtheorem{theorem}{Theorem}
\newtheorem{lemma}[theorem]{Lemma}
\newtheorem{corollary}[theorem]{Corollary}
\newtheorem{proposition}[theorem]{Proposition}
\newtheorem{conjecture}[theorem]{Conjecture}
\newtheorem{definition}[theorem]{Definition}

\theoremstyle{definition}
\newtheorem{remark}[theorem]{Remark}
\newtheorem{example}[theorem]{Example}

\usepackage[all]{xy}
\usepackage[active]{srcltx}
\usepackage[parfill]{parskip}
\usepackage{enumerate}

\usepackage{hyperref}

\newcommand{\tto}{\twoheadrightarrow}

\newcommand{\tens}[1]{\mathbin{\mathop{\otimes}\displaylimits_{#1}}}

\newcommand{\quotient}[2]{\raisebox{.15em}{$#1$}\left/\raisebox{-.15em}{$#2$}\right.}

\newcommand{\form}[2]{\langle #1, #2 \rangle}

\usepackage{tikz}
\usetikzlibrary{fit,matrix,positioning,calc}
\usepackage{tkz-graph}

\newcommand{\C}{\mathbb{C}}
\newcommand{\D}{\mathbb{D}}
\newcommand{\Z}{\mathbb{Z}}
\newcommand{\LL}{\mathfrak{L}}
\newcommand{\T}{\mathfrak{T}}
\newcommand{\g}{\mathfrak{g}}
\newcommand{\h}{\mathfrak{h}}
\newcommand{\n}{\mathfrak{n}}
\newcommand{\bb}{\mathfrak{b}}
\newcommand{\rr}{\mathfrak{r}}
\newcommand{\s}{\mathfrak{s}}
\newcommand{\HH}{\mathfrak{H}}
\DeclareMathOperator{\Hom}{Hom}
\DeclareMathOperator{\End}{End}
\DeclareMathOperator{\Ext}{Ext}
\DeclareMathOperator{\Ker}{Ker}
\DeclareMathOperator{\Ann}{Ann}
\DeclareMathOperator{\Ind}{Ind}
\DeclareMathOperator{\Rad}{Rad}
\DeclareMathOperator{\Nrad}{Nrad}
\DeclareMathOperator{\ad}{ad}
\DeclareMathOperator{\Sym}{Sym}

\DeclareMathOperator{\EnAr}{\mathbf{EA}}
\DeclareMathOperator{\Com}{\mathbf{C}}

\DeclareMathOperator{\GK}{GK}
\DeclareMathOperator{\e}{e}

\let\originalleft\left
\let\originalright\right
\renewcommand{\left}{\mathopen{}\mathclose\bgroup\originalleft}
\renewcommand{\right}{\aftergroup\egroup\originalright}

\begin{document}
\title[Lie algebra modules which are locally finite over the semi-simple part]
{Lie algebra modules which are locally finite\\ over the semi-simple part}

\author{Volodymyr Mazorchuk and Rafael Mr{\dj}en}

\begin{abstract}
For a finite-dimensional Lie algebra $\mathfrak{L}$ over $\mathbb{C}$ with a fixed Levi decomposition $\mathfrak{L} = \mathfrak{g} \ltimes \mathfrak{r}$ where $\mathfrak{g}$ is semi-simple, we investigate $\mathfrak{L}$-modules which decompose, as $\mathfrak{g}$-modules, into a direct sum of simple finite-dimensional $\mathfrak{g}$-modules with finite multiplicities. We call such modules $\mathfrak{g}$-Harish-Chandra modules. We give a complete classification of simple $\mathfrak{g}$-Harish-Chandra modules for the Takiff Lie algebra associated to $\mathfrak{g} = \mathfrak{sl}_2$, and for the Schr\"{o}dinger Lie algebra, and obtain some partial results in other cases. An adapted version of Enright's and Arkhipov's completion functors plays a crucial role in our arguments. Moreover, we calculate the first extension groups of infinite-dimensional simple $\mathfrak{g}$-Harish-Chandra modules and their annihilators in the universal enveloping algebra, for the Takiff $\mathfrak{sl}_2$ and the Schr\"{o}dinger Lie algebra. In the general case, we give a sufficient condition for the existence of infinite-dimensional simple $\mathfrak{g}$-Harish-Chandra modules.
\end{abstract}

\maketitle

\renewcommand{\contentsname}{}
\tableofcontents

\section{Introduction and description of the results}\label{s1}

Construction and classification of modules are two fundamental problems in representation theory.
In most of the cases, the problem of classification of all modules is known to be very hard (i.e. wild),
so one is naturally forced to consider special classes of modules, for example, simple modules.
Classification of simple modules is also quite hard in most of the cases. For example, for complex 
semi-simple Lie algebras,
classification of all simple modules is only known, in some sense, for the algebra $\mathfrak{sl}_2$,
see \cite{Bl}. At the same time, numerous families of simple modules for semi-simple Lie algebras
are very well understood, for example, simple highest weight modules, see \cite{Ve,humphreys2008representations},
Whittaker modules, see \cite{kostant1978on,BM}, weight modules with finite-dimensional weight spaces,
see \cite{mathieu2000classification}, and Gelfand-Zeitlin modules, see \cite{DOF,EMV,We} 
and references therein.

There are two natural generalizations of semi-simple Lie algebras: semi-simple Lie superalgebras
and non-semi-simple Lie algebras. For semi-simple Lie superalgebras, a significant progress in 
classification of simple modules was recently made in \cite{CM,CCM}. Basically, for a large class
of Lie superalgebras, the results of \cite{CM,CCM} reduce the problem of classification of simple modules
to a similar problem for the even part of the superalgebra, which is a reductive Lie algebra.
In contrast to this situation, for non-semi-simple Lie algebras, very little is know.
Apart from the main result of \cite{Bl}, which, in addition to $\mathfrak{sl}_2$,
classifies simple modules over the Borel subalgebra of $\mathfrak{sl}_2$,
several special classes of simple modules were studied for various specific non-semi-simple Lie algebras,
see e.g. \cite{dubsky2014category,wilson2011highest,cai2014quasi,cai2017quasi,lau2018classification,bavula2017prime,bavula2018classification,mazorchuk2019category}
and references therein. We will now look at some of these and some other results in more detail.

It seems that the so called current Lie algebras are the ones which are most studied and best understood.
These are defined as tensor product of a reductive Lie algebra with a commutative unital associate algebra. 
For current Lie algebras, there is a full classification of simple weight modules with finite-dimensional 
weight spaces, see \cite{lau2018classification}. Also, the highest weight theory 
for the truncated polynomial version of these Lie algebras is developed in \cite{wilson2011highest}. 
Moreover, the center of the universal enveloping algebras of such Lie algebras whose semi-simple part is 
of type $A$ is described explicitly in \cite{molev1996casimir}.

A special family of truncated current Lie algebras is formed by the so called Takiff Lie algebras, studied originally in \cite{takiff1971rings}, which correspond to the case when one tensors a reductive Lie algebra with the associative algebra of dual numbers. The Takiff $\mathfrak{sl}_2$ is also known as the complexification of the Lie algebra of the Euclidean group $E(3)$, the Lie group of orientation-preserving isometries of the three-dimensional Euclidean space. It belongs to the family of conformal Galilei algebras, see e.g. \cite{lu2014simple}. Category $\mathcal{O}$ for Takiff $\mathfrak{sl}_2$ was recently studied in \cite{mazorchuk2019category} and simple weight modules were classified in \cite{bavula2017prime}.

The Schr\"{o}dinger Lie algebra, see Section \ref{section:schrodinger}, is also an important and intensively studied example of a non-reductive Lie algebra. Its category  $\mathcal{O}$ was studied in detail in \cite{dubsky2014category}, lowest weight modules were classified in \cite{dobrev1997lowest}, and simple weight modules were classified in \cite{dubsky2014classification, bavula2018classification}.

A slight modification of the Schr\"{o}dinger Lie algebra, called the centerless Schr\"{o}dinger  Lie algebra belongs to the family of conformal Galilei algebras, see Section \ref{section:conf_galilei}. As their names suggest, the Schr\"{o}dinger Lie algebra and conformal Galilei algebras are of great importance in theoretical physics and seem to have originated from there. For example, the Schr\"{o}dinger Lie algebra comes from the Schr\"{o}dinger Lie group, the group of symmetries of the free particle Schr\"{o}dinger equation, see \cite{dobrev1997lowest, perroud1977projective}. Conformal Galilei algebras are related to the non-relativistic version of the AdS/CFT correspondence, see \cite{bagchi2009galilean}.

Several papers studied a generalization of Whittaker modules (originally defined in \cite{kostant1978on} for semi-simple Lie algebras), in the setup of 
conformal Galilei algebras and the Schr\"{o}dinger Lie algebra,
see \cite{cai2014quasi,cai2016whittaker,lu2014simple, cai2017quasi}. 
Quasi-Whittaker modules are modules on which the radical of the Lie algebra
acts locally finitely.

In the present  paper we initiate the study of modules over (non-semi-simple) 
Lie algebras on which the action of the semi-simple part of the Lie algebra
is locally finite, that is, which are locally finite over the semi-simple part. 
This condition is, in a sense, the opposite to the condition defining 
quasi-Whittaker modules. 
The obvious examples of modules that are locally finite over the semi-simple part
are simple finite-dimensional modules
over the semi-simple part on which the radical of our Lie algebra acts trivially.
However, we observe that, for many Lie algebras, there exist simple infinite-dimensional modules that are locally finite over the semi-simple part.
This motivates the problem of classification of such modules, 
and we show that this problem can be completely answered 
for the Takiff Lie algebra of $\mathfrak{sl}_2$
and for the Schr\"{o}dinger Lie algebra. Moreover, the answer is both
non-trivial and interesting. To our best knowledge, such modules 
have not been studied before in the general case (however, for the Schr{\"o}dinger
Lie algebra and the Takiff $\mathfrak{sl}_2$, they belongs to a larger family of 
weight modules studied in detail in \cite{bavula2017prime,bavula2018classification}). 
Let us now describe the content of the paper in more detail.

If $\LL$ is any finite-dimensional Lie algebra and $\g \subseteq \LL$ its semi-simple Levi subalgebra, we study $\LL$-modules whose restriction to $\g$ decomposes into a direct sum of simple finite-dimensional $\g$-modules with finite multiplicities, and call them {\em $\g$-Harish-Chandra modules}. To justify the name, we note that there is an obvious analogy with the classical theory of $(\g,K)$-modules as in \cite{vogan1981repersentations}, coming from the setup of real reductive Lie groups. In the classical theory, any $(\g,K)$-module splits as a direct sum of finite-dimensional modules over the compact group $K$, and moreover, the multiplicities are finite if the corresponding group representation is irreducible and unitary (a result by Harish-Chandra). In our setup there is no such automatic splitting, so we ``pretend'' that $\g$ is compact, i.e., we consider only those $\LL$-modules that split as $\g$-modules into a direct sum of finite-dimensional $\g$-modules with finite multiplicities. Hopefully, this analogy could be used to transfer parts of the Langlands classification, or the theory of minimal $K$-types into our non-reductive algebraic setup.
There is another analogy of our setup with integrable modules over a Kac-Moody algebra, see \cite{kac1990infinite}.

In Section~\ref{s2}, we introduce the basic setup that we work in. 
In Section~\ref{section:Takiff}, we roughly describe ``universal'' $\g$-Harish-Chandra modules for Takiff Lie algebras. In particular, we show that such Lie algebras do indeed always have simple infinite-dimensional $\g$-Harish-Chandra modules, see Corollary~\ref{corollary:uniqueness_with_0}.

In Sections~\ref{section:takiff_sl2} and \ref{section:schrodinger} we prove our
most concrete results: Theorem~\ref{theorem:classification} provides a
complete classification of simple $\g$-Harish-Chandra modules for the 
Takiff $\mathfrak{sl}_2$, and Theorem~\ref{theorem:classification_sch}
gives such a classification for the Schr\"{o}dinger Lie algebra.
These two answers have both clear similarities and differences.
In both cases we crucially use the highest weight theory for
corresponding algebras and appropriate analogues of completions functors.
Also, in both cases, we can consider semi-simple
$\g$-Harish-Chandra modules as a monoidal representation of the
monoidal category of finite-dimensional $\mathfrak{sl}_2$-modules.
We found it surprizing that the combinatorial properties of the
corresponding monoidal representation in the Takiff $\mathfrak{sl}_2$ and the Schr\"{o}dinger
cases are rather different.

In case of the Takiff $\mathfrak{sl}_2$, we obtain a family of modules 
$V(n,\chi)$ which are naturally parameterized by $n \in \Z$ and 
$\chi \in \C\setminus\{0\}$. However, we show that this family has a redundancy
via non-trivial isomorphisms $V(n,\chi) \cong V(-n,-\chi)$. 
Roughly speaking, $|n|$ is the minimal $\g$-type, and $\chi^2$ is the 
``purely radical part'' of the central character. This classifies all 
simple infinite-dimensional $\g$-Harish-Chandra modules.

In case of  the Schr\"{o}dinger Lie algebra, we obtain a similar family 
of modules $V(n,\chi)$ parameterized by $n \in \Z_{\geq 0}$ and 
$\chi \in \C\setminus\{0\}$, However, in contrast to the Takiff case,
this family is irredundant.

The modules mentioned above are very explicitly described. In both cases,
we, moreover, show that all groups of first self extensions of these modules 
are one-dimensional, see Theorems~\ref{theorem:extensions} and \ref{theorem:extensions_sch}. 
Additionally, we prove that the annihilators of all the above modules in the universal enveloping algebra are centrally generated, see Theorems~\ref{theorem:annihilators} 
and Corollary \ref{corollary:annihilators_sch}. Classification results in
Sections~\ref{section:takiff_sl2} and \ref{section:schrodinger}
are deducible (with non-trivial effort) from more general results of \cite{bavula2017prime,bavula2018classification}, however, we provide a completely
different, less computational and more conceptual approach.

For comparison, it is easy too see that the centerless Schr\"{o}dinger Lie algebra does 
not admit simple infinite-dimensional $\g$-Harish-Chandra modules. 
Roughly speaking,  because its purely radical part of the center is 
trivial, see Remark~\ref{remark:centerless} for details.

We would like to point out that the methods we utilize for our classification 
go far beyond direct calculations. We use various functorial constructions, which
include, in particular, an appropriate adjustment of Enright's completion functor
(based on Arkhipov's twisting functor), 
\cite{enright1979on, deodhar1980on, konig2002enrights, andersen2003twisting, arkhipov2004algebraic, khomenko2005on}. Further development of both, highest weight 
theory and properties of various Lie theoretic functors as in \cite{MS}, for 
non-semi-simple Lie algebras, should provide an opportunity for generalization of
the results of this paper to, in the first step, other Takiff Lie algebras and,
further, general finite-dimensional Lie algebras. 

In the most general case of an arbitrary finite-dimensional complex Lie algebra $\LL$
and a non-trivial Levi subalgebra $\mathfrak{g}$, it is clear that simple $\g$-Harish-Chandra modules always exist. Namely, the finite-dimensional $\LL$-modules are, of course,
$\g$-Harish-Chandra modules. In Theorem~\ref{theorem:exist_inf_dim_HC} of
Section~\ref{section:existence_general} we give a general sufficient condition 
for existence of infinite-dimensional simple $\g$-Harish-Chandra modules. 
The sufficient condition, as we formulate it, requires that the nilradical of $\LL$ intersects the centralizer in $\LL$
of the Cartan subalgebra of $\mathfrak{g}$. In this case we manage to use highest weight theory for $\LL$,
combined with various versions of twisting functors, to construct 
infinite-dimensional simple $\g$-Harish-Chandra modules. 
We also  provide an example showing that our 
sufficient condition is not necessary, in general: 
the semi-direct product of $\mathfrak{sl}_2$ and its $4$-dimensional simple
module does not satisfy our sufficient condition and has trivial highest
weight theory in the sense that its simple highest weight modules coincide with simple 
highest weight $\mathfrak{sl}_2$-modules. However, using various combinatorial 
tricks from \cite{hahn2018from}, 
we show that this Lie algebra does admit simple infinite-dimensional 
$\g$-Harish-Chandra modules. This result can be found in Subsection~\ref{subscounter}. 

Finally, in Subsection~\ref{subsec7.3}, in particular Theorem \ref{thmconf4new}, we classify a class of 
$\mathfrak{sl}_2$-Harish-Chandra modules that are connected to highest weight modules,
for the semi-direct product of $\mathfrak{sl}_2$ with its simple $5$-dimensional
module. The corresponding category of semi-simple
$\g$-Harish-Chandra modules is, again, a monoidal representation of the
monoidal category of finite-dimensional $\mathfrak{sl}_2$-modules.
But the combinatorics of this monoidal representation 
is completely different from the ones which we get
in the Takiff and the Schr\"{o}dinger cases, see Remark \ref{remark:comb_struct_conf4}. In particular, contrary to the previous cases, in this case we obtain an example of two simple $\g$-Harish-Chandra modules with different sets of $\g$-types, but with the same minimal $\g$-type.

Comparison of the results of \cite{han2019higher}
with Lemmata \ref{lemma:structure_Q0} and \ref{lemma:structure_Q0_sch} suggests a possibility
of an interesting connection between $\g$-Harish-Chandra
modules and higher-spin algebras from \cite{pope1990W_infinity}.

\subsection*{Acknowledgments}
This research was partially supported by
the Swedish Research Council, G{\"o}ran Gustafsson Stiftelse and Vergstiftelsen. 
R. M. was also partially supported by the QuantiXLie Center of Excellence grant 
no. KK.01.1.1.01.0004 funded by the European Regional Development Fund.

\section{Notation and preliminaries}\label{s2}

We work over the complex numbers $\C$. For a Lie algebra $\mathfrak{a}$,
we denote by $U(\mathfrak{a})$ the universal enveloping algebra of $\mathfrak{a}$.

Fix a finite-dimensional Lie algebra $\LL$ over $\C$, and fix its Levi decomposition $\LL \cong \g \ltimes \rr$. This is a semi-direct product, where $\g$ is a maximal semi-simple Lie subalgebra, unique up to conjugation, and $\rr = \Rad \LL$ is the radical of $\LL$ , i.e., the unique maximal solvable ideal.

\begin{definition}
An $\LL$-module $V$ is called \emph{$\g$-Harish-Chandra} module, if
the restriction of $V$ to $\g$ decomposes as a direct sum of simple finite-dimensional $\g$-modules, and moreover, each isomorphism class of simple finite-dimensional $\g$-modules occurs  with a finite multiplicity in $V$.
\end{definition}

A simple $\g$-submodule of a $\g$-Harish-Chandra module $V$ is called a \emph{$\g$-type} of $V$. The sum of all $\g$-submodules of $V$ isomorphic to a given $\g$-type is called the \emph{$\g$-isotypic component} of $V$ determined by this $\g$-type.

Fix a Cartan subalgebra $\h \subseteq \g$. Every $\g$-Harish-Chandra module 
is a weight module with respect to $\h$. However, infinite-dimensional
$\g$-Harish-Chandra modules might have infinite-dimensional weight spaces.

\begin{remark}
Note that the notion of a $\g$-Harish-Chandra module is different from the 
notion of Harish-Chandra module from \cite{lau2018classification}. In the latter paper, 
Harish-Chandra modules are weight modules with finite-dimensional weight spaces. 
It would be natural to call the modules from \cite{lau2018classification}
$\h$-Harish-Chandra modules.
\end{remark}

Denote by $\Nrad(\LL)$ the \emph{nilradical} of $\LL$, by which we mean the intersection of kernels of all finite-dimensional simple modules of $\LL$. It is a nilpotent ideal, but not necessarily equal to the maximal nilpotent ideal. It is well known that $\Nrad(\LL) = [\LL,\LL] \cap \rr = [\LL,\rr]$, and $\LL$ is reductive if and only if $\Nrad(\LL) =0$. Moreover, $\Nrad(\LL)$ is the minimal ideal in $\LL$ for which the quotient $\quotient{\LL}{\Nrad(\LL)}$ is reductive. For proofs, see e.g. \cite[Chapter I, \S 5.3.]{bourbaki1989lie}.

\begin{example}
If $\LL=\g \ltimes \rr$ is a reductive Lie algebra, then $\rr$ is precisely the center of $\LL$. If $V$ is a simple $\g$-Harish-Chandra module for $\LL$, then, by Schur's lemma, $\rr$ acts by scalars on $V$. It follows that $V$ is just a simple finite-dimensional $\LL$-module. So, the notion of $\g$-Harish-Chandra modules in not very interesting for reductive Lie algebras.
\end{example}

Fix a positive part $\Delta^+(\g,\h)$ in the root system $\Delta(\g,\h)$, and a non-degenerate invariant symmetric bilinear form $\form{-}{-}$ on $\h^\ast$. We have the classical triangular decomposition $\g = \n_- \oplus \h \oplus \n_+$. Further, fix a weight $\delta \in \h^\ast$
such that $\form{\delta}{\alpha}>0$ for all $\alpha\in \Delta^+(\g,\h)$
and such that $\form{\delta}{\alpha}=0$, for an integral weight $\alpha$, implies $\alpha=0$. Since $\LL$ is a finite-dimensional $\g$-module with respect to the adjoint action, it decomposes as a direct sum of its weight spaces $\LL_\mu$, where $\mu$ varies over the set of integral weights in $\h^\ast$. Consider the following Lie subalgebras of $\LL$:
\begin{equation} \label{equation:general_triangular}
\tilde{\n}_- := \bigoplus_{\langle\mu ,\delta \rangle < 0} \LL_\mu, \quad \tilde{\h} := \bigoplus_{\langle\mu ,\delta \rangle = 0} \LL_\mu \quad \tilde{\n}_+ := \bigoplus_{\langle\mu ,\delta \rangle > 0} \LL_\mu.
\end{equation}
Note that this decomposition heavily depends on the choice of $\delta$
and not only on the choice of $\Delta^+(\g,\h)$. However, for example, for
truncated current Lie algebras (which include Takiff Lie algebras), the Schr\"{o}dinger Lie algebra and conformal Galilei algebras, the decomposition \eqref{equation:general_triangular} only depends on the choice 
of $\Delta^+(\g,\h)$.
From the construction, it is clear that $\tilde{\n}_\pm \cap \g = \n_\pm$, and $\tilde{\h} \cap \g = \h$. Moreover, from the condition prescribed on $\delta$, it follows that $\tilde{\h}$ is precisely the centralizer of $\h$ in $\LL$. The decomposition $\LL = \tilde{\n}_- \oplus \tilde{\h} \oplus \tilde{\n}_+$ does not satisfy, in general, all the axioms in \cite[Section~2]{wilson2011highest}, since we do not require existence of any analogue of Chevalley involution (and even the dimensions
of $\tilde{\n}_-$ and $\tilde{\n}_+$ might be different). However, it is good enough to define Verma modules with reasonable properties. 

For an element $\lambda \in \tilde{\h}^\ast$, denote the one-dimensional $\tilde{\bb}:=\tilde{\h} \oplus \tilde{\n}_+$-module where $\tilde{\h}$ acts as $\lambda$ and $\tilde{\n}_+$ acts trivially, by $\C_\lambda$. The \emph{Verma module} with highest weight $\lambda$ is defined as
\begin{equation} \label{equation:Verma_def}
\Delta(\lambda) := \Ind_{\tilde{\bb}}^\LL \C_\lambda = U(\LL) \tens{U(\tilde{\bb})} \C_\lambda \cong U(\tilde{\n}_-) \tens{\C} \C_\lambda.
\end{equation}
Let $\Delta^{\pm}(\LL,\h)$ denote the set of all $\mu$ such that $\LL_{\mu}\neq 0$
and $\pm\form{\mu}{\delta}>0$. We also set 
$\Gamma^{\pm}=\mathbb{Z}_{\geq 0}\Delta^{\pm}(\LL,\h)$.
Recall that the {\em support} of a weight module is the set of 
all weights for which the corresponding weight spaces are non-zero.
By the standard arguments using PBW theorem (cf. \cite{humphreys2008representations}), we have:
\begin{proposition} \label{proposition:Vermas_general}
The Verma module $\Delta(\lambda)$ is an $\h$-weight module, 
whose $\h$-support is $\lambda|_{\h}+\Gamma^{-}$. The $\lambda|_{\h}$-weight space is one-dimensional, and $\Delta(\lambda)$ is generated by this weight vector, so any non-trivial quotient of $\Delta(\lambda)$ also has one-dimensional $\lambda|_{\h}$-weight space. Moreover, $\Delta(\lambda)$ has a unique simple quotient, which we denote by $\mathbf{L}(\lambda)$.
\end{proposition}
For $\lambda \in \h^\ast$, we denote by $\Delta^\g(\lambda)$ the classical Verma 
module for $\g$ with highest weight $\lambda$ with respect to $\Delta^+(\g,\h)$, 
and by $L(\lambda)$ the unique simple quotient of $\Delta^\g(\lambda)$.

\section{$\g$-Harish-Chandra modules for Takiff Lie algebras}
\label{section:Takiff}

\subsection{Setup}

Fix a finite-dimensional semi-simple Lie algebra $\g$ over $\C$. Define the associated \emph{Takiff Lie algebra} $\T$ as \[ \T := \g \tens{} \D, \] where $\D = \C[x] / (x^2)$ is the algebra of dual numbers. The Lie bracket of $\T$ is defined in the following way:
\[ [v \otimes x^i, w \otimes x^j ] := [v,w] \otimes x^{i+j}. \]
We identify $\g$ with the subalgebra $\g \otimes 1 \subseteq \T$, and denote by $\bar{\g} = \g \otimes x \subseteq \T$. Then $\bar{\g}$ is a commutative ideal in $\T$, and $\T \cong \g \ltimes \bar{\g}$ (the semi-direct product given by the adjoint action of $\g$ on $\bar{\g}$). For $v \in \g$, we denote by $\bar{v} = v \otimes x \in \bar{\g}$.

Observe that the nilradical of $\T$ is $\Nrad(\T) = [\T,\bar{\g}]=\bar{\g}$. This means that $\bar{\g}$ must necessarily annihilate any simple finite-dimensional $\T$-module.

In the triangular decomposition (\ref{equation:general_triangular}) for $\T$, we have $\tilde{\h} = \h \oplus \bar{\h}$ and $\tilde{\n}_\pm = \n_\pm \oplus \bar{\n}_\pm$.  We want to note that this is also a triangular decomposition in the sense of \cite{wilson2011highest}. A simplicity criterion for Verma modules over $\T$ can be found in \cite[Theorem~7.1]{wilson2011highest}.

\subsection{Purely Takiff part of the center}\label{spureT}

The universal enveloping algebra $U(\T)$ is free as a module over its center $Z(\T)$, see \cite{geoffriau1994centre, geoffriau1995homomorphisme, futorny2005kostant}. In case $\g$ is of type $A$, algebraically independent generators of the center are given explicitly in \cite{molev1996casimir}.

\begin{proposition}
There is an isomorphism of algebras
\begin{equation} \label{equation:Takiff_center}
Z(\g) \cong Z(\T) \cap U(\bar{\g}) .
\end{equation}
\end{proposition} 
\begin{proof}
This is clear since $U(\g) \cong U(\bar{\g})$ as $\g$-modules with respect to the adjoint action. By taking $\g$-invariants, we get (\ref{equation:Takiff_center}).
\end{proof}

It is easy to see that the isomorphism can be obtained by putting bars on all Lie algebra elements that appear in an expression in a fixed PBW-basis of elements from $Z(\g)$. Hence we denote the right-hand side of (\ref{equation:Takiff_center}) by $\overline{Z(\g)}$. This will be referred to as the \emph{purely Takiff part of the center} $Z(\T)$. The full center $Z(\T)$ is in general bigger than $\overline{Z(\g)}$, see \cite{molev1996casimir}.

\subsection{Universal modules}

Assume that the simple $\g$-module $L(\lambda)$ is finite-dimensional,
that is $\lambda \in \h^\ast$ is integral, dominant and regular. Define
\begin{equation*}
Q(\lambda) := \Ind_{\g}^{\T} L(\lambda) = U(\T) \tens{U(\g)} L(\lambda) \cong U(\bar{\g}) \tens{\C} L(\lambda).
\end{equation*}

Note that we have (see e.g. \cite[Proposition 6.5.]{knapp1988lie})
\begin{equation} \label{equation:tensor}
Q(\lambda) \cong Q(0) \tens{} L(\lambda),
\end{equation}
where we consider $L(\lambda)$ as a $\T$-module with the trivial $\bar{\g}$-action. 

\begin{proposition} \label{proposition:end_Q0}
We have the following isomorphism of algebras:
\begin{equation*}
\End(Q(0)) \cong \overline{Z(\g)}.
\end{equation*}
\end{proposition}
\begin{proof}
The module $Q(0)$ is generated by $1 \otimes 1$ by construction, so any 
endomorphism of $Q(0)$ is uniquely determined by the image of $1 \otimes 1$. 
Denote this image by $u \otimes 1$, for some $u \in U(\bar{\g})$. 
The element $u\otimes 1$ generates the trivial $\g$-submodule (since $1\otimes 1$ does), 
so $u$ must commute with $\g$. Of course, $u$ commutes with $\bar{\g}$. Hence 
$u \in Z(\T) \cap U(\bar{\g}) = \overline{Z(\g)}$.

Conversely, any $u \in \overline{Z(\g)}$, being central, defines an endomorphism 
of $Q(0)$. This endomorphism maps $1 \otimes 1$ to $u \otimes 1$.
The claim follows.
\end{proof}

For an algebra homomorphism $\chi \colon \overline{Z(\g)} \to \C$, 
consider the corresponding \emph{universal module} 
\begin{equation*}
Q(\lambda,\chi) := \quotient{Q(\lambda)}{\mathbf{m}_\chi Q(\lambda)},
\end{equation*}
where $\mathbf{m}_\chi$ is the maximal ideal in $\overline{Z(\g)}$ corresponding to $\chi$. 
On $Q(\lambda,\chi)$, the purely Takiff part of the center acts via the scalars
prescribed by $\chi$. Observe that from \eqref{equation:tensor}
and the right-exactness of tensor product we have
\begin{equation} \label{equation:tensor2}
Q(\lambda,\chi) \cong Q(0,\chi) \tens{} L(\lambda).
\end{equation}

For finite-dimensional simple $\g$-modules $L(\mu)$, $L(\nu)$ and $L(\lambda)$, denote by $l_{\nu,\lambda}^\mu$ the Littlewood-Richardson coefficient, i.e., the multiplicity of $L(\mu)$ in $L(\nu) \otimes L(\lambda)$.

\begin{proposition} \label{proposition:multiplicity}
\begin{enumerate}[(a)]
\item Let $\lambda$, $\chi$ be as before. The module $Q(\lambda,\chi)$ is a $\g$-Harish-Chandra module, and the multiplicities are given as follows:
\begin{equation} \label{equation:multiplicity}
[Q(\lambda,\chi) \colon L(\mu)] = \sum_{\nu} \dim L(\nu)_0  \cdot l_{\nu,\lambda}^\mu < \infty.
\end{equation}

\item \label{item:universal_property}
Let $V$ be any simple $\T$-module that has some finite-dimensional $L(\lambda)$ as a simple $\g$-submodule. Then $V$ is a quotient of $Q(\lambda,\chi)$ for a unique $\chi$. In particular, $V$ is a $\g$-Harish-Chandra module, and (\ref{equation:multiplicity}) gives an upper bound for the multiplicities of its $\g$-types.

\end{enumerate}
\end{proposition}
\begin{proof}
\begin{enumerate}[(a)]

\item Suppose first that $\lambda=0$. Then, as a $\g$-module, $Q(0)$ is isomorphic to $U(\g)$ with respect to the adjoint action. Taking the $\chi$-component of $Q(0)$ corresponds to factoring $U(\g)$ by the ideal generated by the corresponding central character of $Z(\g)$. From Kostant's theorem (see \cite[Section~3.1]{jantzen11983einhullende}), it follows that 
$Q(0)$ decomposes as direct sum of finite-dimensional $\g$-submodules, and that $[Q(0,\chi) \colon L(\mu)] = \dim L(\mu)_0$.
The general statement now follows from (\ref{equation:tensor2}).

Note that the value in (\ref{equation:multiplicity}) is finite, since, for fixed
$\mu$ and $\lambda$, the value $l^{\mu}_{\nu,\lambda}$ is non-zero only for finitely 
many $\nu$.

\item
This follows from Schur's Lemma by adjunction. \qedhere
\end{enumerate} 
\end{proof}

\begin{corollary} \label{corollary:uniqueness_with_0}
Given $\chi$, there exists a unique simple $\T$-module $V$ which contains $L(0)$ as a $\g$-submodule and has the Takiff part of the central character equal to $\chi$. Moreover, $V$ is a $\g$-Harish-Chandra module.

Furthermore, if $\chi$ does not correspond to the trivial $\T$-module, then $V$ is infinite-dimensional.
\end{corollary}
\begin{proof}
By Proposition \ref{proposition:multiplicity}, the module $Q(0,\chi)$ has a unique occurrence of $L(0)$, and is generated by it. Therefore, the sum all its submodules not containing $L(0)$ as a composition factor is the unique maximal submodule; denote it by $N$. It follows that $V := \quotient{Q(0,\chi)}{N}$ is the unique simple quotient of $Q(0,\chi)$.

Suppose now that $V$ is finite-dimensional. The nilradical $\bar{\g}$ must act trivially on it. Because of simplicity, we must have $V=L(0)$, which is a contradiction.
\end{proof}

\begin{conjecture} \label{conjecture:Takiff_simple}
For a ``generic'' $\chi$, the module $Q(0,\chi)$ is simple.
\end{conjecture}
We will prove this conjecture for the Takiff $\mathfrak{sl}_2$ case in Section \ref{section:takiff_sl2}. We will also prove it for the Schr\"{o}dinger Lie algebra in Section \ref{section:schrodinger} (but, strictly speaking, it is not a special instance of the above conjecture). This is the starting point in our classification of $\g$-Harish-Chandra modules  for these Lie algebras.

\section{$\mathfrak{sl}_2$-Harish-Chandra modules for the Takiff $\mathfrak{sl}_2$}
\label{section:takiff_sl2}

\subsection{Setup}

For this section, we fix the Takiff Lie algebra associated to $\g := \mathfrak{sl}_2$:
\[ \T = \mathfrak{sl}_2 \tens{} \D = \mathfrak{sl}_2 \ltimes \mathfrak{\overline{sl}}_2. \]
We use the usual notation $f,h,e$ for the standard basis elements of $\mathfrak{sl}_2$, and $\bar{f}, \bar{h}, \bar{e}$ for their counterparts in the ideal $\mathfrak{\overline{sl}}_2$.

Our classification of simple $\g$-Harish-Chandra modules for the Takiff 
$\mathfrak{sl}_2$ should be, of course, deducible from the classification of 
all simple weight modules  given in \cite{bavula2017prime}. However, our
approach is completely different and, unlike the approach of \cite{bavula2017prime},
has clear potential for generalization to other Lie algebras. Also,
our description of simple $\g$-Harish-Chandra modules is much more explicit, and
it provides a connection to highest weight theory for $\T$ and utilizes the
use of analogues of projective functors for $\T$.

The center $Z(\T)$ is a polynomial algebra generated by two algebraically independent elements (see \cite{molev1996casimir}):
\begin{align} \label{equation:Z_T}
& C = \bar{h}h + 2\bar{f}e + 2\bar{e}f, \\
\nonumber & \bar{C} = \bar{h}^2 + 4\bar{f} \bar{e}.
\end{align}
The purely Takiff part of the center is, of course, $\overline{Z(\g)} = \C[\bar{C}]$. So, a homomorphism $\chi \colon \overline{Z(\g)} \to \C$ is uniquely determined by the value $\chi(\bar{C})$, which can be an arbitrary complex number. In the remainder, we write
$\chi$ for $\chi(\bar{C})$, for the sake of brevity.

\subsection{Universal modules}

We can describe $Q(0,\chi)$ very explicitly.

\begin{lemma} \label{lemma:structure_Q0}
\begin{enumerate}[(a)]
\item As $\T$-modules, we have $Q(0) \cong U(\bar{\g})$ and $Q(0,\chi) \cong \quotient{U(\bar{\g})}{(\bar{C}-\chi)}$, where $\g$ acts by the adjoint action, 
and $\bar{\g}$ by the left multiplication.

\item \label{item:basis} The set $\big\{\bar{f}^i \bar{h}^\epsilon \bar{e}^j \colon i,j \geq 0, \ \epsilon \in \{0,1\}\big\}$ is a basis for $Q(0,\chi)$.

\item As a $\g$-module, $Q(0,\chi) \cong \bigoplus_{k\geq 0} L(2k)$. Moreover, $\bar{e}^k$ is the highest weight vector in $L(2k)$.

\item $C$ acts as zero on $Q(0)$ and on every $Q(0,\chi)$.
\end{enumerate}
\end{lemma}
\begin{proof}
The first claim is clear. The second one follows from the PBW basis in $U(\g)$ and the relation $\bar{h}^2 = -4\bar{f} \bar{e} + \chi$ in the quotient.

The decomposition in the third claim is given by Kostant's theorem, see 
\cite[Section~3.1]{jantzen11983einhullende}.  
Since $\bar{e}^k$ is of weight $2k$ and annihilated by $e$, 
it must be highest weight vector of a $\g$-submodule isomorphic to $L(2k)$, which, we know, occurs uniquely in $Q(0,\chi)$.

The last claim follows from the definitions by a direct calculation.
\end{proof}

The action of $\T$ on $U(\bar{\g})$ and its quotients will be denoted by $\circ$, in order not to confuse it with the multiplication $\cdot$ in the enveloping algebra. These coincide for $\bar{\g}$ but not for $\g$, where the action is adjoint. Note that $U(\bar{\g})$ is not closed under the left multiplication with the whole $\T$.

\begin{theorem} \label{theorem:Q_simple}
The module $Q(0,\chi)$ is simple if and only if $\chi \neq 0$.

The module $Q(0,0)$ has infinite length, and a $\T$-filtration whose composition factors are $L(0), L(2), L(4) \ldots$ with the trivial action of $\bar{\g}$. 
\end{theorem}
\begin{proof}
Assume $\chi \neq 0$, and let $V \subseteq Q(0,\chi)$ be non-zero submodule. Take $k$ to be the smallest non-negative integer such that $L(2k) \subseteq V$. If $k=0$, then $V = Q(0,\chi)$ since $L(0)$ generates $Q(0,\chi)$, and we are done. So, let us assume now $k \geq 1$. We have $\bar{e}^k \in V$, so if we find an element from $U(\T)$ that maps $\bar{e}^k \in V$ to $\bar{e}^{k-1}$, we will get a contradiction. That element can be taken as $\frac{1}{k\chi}(4k \bar{f} - \bar{h}f)$, namely:
\begin{align*}
(4k \bar{f} - \bar{h}f) \circ \bar{e}^k &= 4k\bar{f}\bar{e}^k - \bar{h}[f,\bar{e}^k] \\
&= 4k\bar{f}\bar{e}^k + k \bar{h}^2\bar{e}^{k-1} \\
&= 4k\bar{f}\bar{e}^k + k (-4\bar{f} \bar{e} + \chi)\bar{e}^{k-1} \\
&= k\chi \bar{e}^{k-1}.
\end{align*}
We conclude that $Q(0,\chi)$ is simple.

For the converse, assume $\chi = 0$. We will show that for any $k \geq 0$, the subspace $Q_k := \oplus_{t \geq k} L(2t)$ is a submodule. From this, the theorem will follow.

Let us first prove that $L(2k)$ is equal to the span of $\big\{\bar{f}^i \bar{h}^\epsilon \bar{e}^j \colon \epsilon \in \{0,1\},  i+\epsilon + j = k \big\}$. This set contains $\bar{e}^k$, so it is enough to see that it is stable under $f$. We calculate the two cases whether $\epsilon$ is $0$ or $1$ separately:
\begin{align*}
f \circ \bar{f}^i \bar{e}^j &= \bar{f}^i [f,\bar{e}^j] = - j \bar{f}^i \bar{h} \bar{e}^{j-1}, \\[.5em]
f \circ \bar{f}^i \bar{h} \bar{e}^j &= \bar{f}^i [f,\bar{h}]\bar{e}^j + \bar{f}^i \bar{h} [f,\bar{e}^j] \\
&= 2\bar{f}^{i+1} \bar{e}^j -j \bar{f}^i \bar{h}^2 \bar{e}^{j-1} \\
&= 2\bar{f}^{i+1} \bar{e}^j +4j \bar{f}^{i+1} \bar{e}^{j} \\
&= (4j+2)\bar{f}^{i+1} \bar{e}^j.
\end{align*}

From this description of $L(2k)$, one easily checks that $\bar{f}, \bar{h},\bar{e}$ map $L(2k)$ to $L(2k+2)$. From this, it follows that $Q_k$ is a submodule.
\end{proof}

\begin{remark}[A sketch of an alternative proof of simplicity of $Q(0,\chi)$ for $\chi\neq 0$]  \label{remark:simpllicity_Q0}
Suppose $V$ is a $\T$-submodule of $Q(0,\chi)$ containing $L(2k)$, with $k >0$ minimal. By applying $\bar{e}$, we see that, as a $\g$-module, $V \cong L(2k) \oplus L(2k+2) \oplus \ldots$. This implies that the quotient
\[ \quotient{Q(0,\chi)}{V} \cong L(0) \oplus L(2) \oplus \ldots \oplus L(2k-2) \]
is simple as a $\T$-module and is finite-dimensional. Since $\bar{C}$ consists of elements from the nilradical of $\T$, it must act as zero on this quotient. This is a contradiction with $\chi \neq 0$.
\end{remark}

To classify simple $\g$-Harish-Chandra modules, by (\ref{equation:tensor2}) and Proposition \ref{proposition:multiplicity}(\ref{item:universal_property}) we should find all simple quotients of all tensor products of $Q(0,\chi)$ with finite-dimensional $\g$-modules. It is not easy to do this directly, so we establish a connection with Verma modules, and perform calculations there.

\subsection{Verma modules}

Verma modules for the Takiff $\mathfrak{sl}_2$ are studied in detail in \cite{mazorchuk2019category}. Recall (\ref{equation:Verma_def}) and Proposition \ref{proposition:Vermas_general}. Also recall that $\tilde{\h} = \h \oplus \bar{\h}$ and $\tilde{\n}_\pm = \n_\pm \oplus \bar{\n}_\pm$. For a weight $\lambda \in \tilde{\h}^\ast = \h^\ast \oplus \bar{\h}^\ast$, we denote $\lambda_1 := \lambda(h)$ and $\lambda_2 := \lambda(\bar{h})$.
\begin{proposition}[{\cite[Proposition 1]{mazorchuk2019category} or \cite[Theorem 7.1]{wilson2011highest}}]
The Verma module $\Delta(\lambda)$ is simple if and only if $\lambda_2 \neq 0$.
\end{proposition}

The generators of the center $C$ and $\bar{C}$ act on the Verma module $\Delta(\lambda)$ as the scalars $\lambda_2(\lambda_1+2)$ and $\lambda_2^2$ respectively, see (\ref{equation:Z_T}). Therefore, with our convection, $\chi = \lambda_2^2$.

\begin{lemma} \label{lemma:Verma_central_ch}
Non-isomorphic Verma modules $\Delta(\lambda)$ and $\Delta(\lambda')$ have the same central character if and only if either $\lambda'_2=\lambda_2=0$, or $\lambda'_2=-\lambda_2 \neq 0$ and $\lambda'_1 = -\lambda_1-4$.
\end{lemma}
\begin{proof}
From the explicit description of generators of the center, we get a system of equations
\[ \begin{cases} \lambda_2(\lambda_1+2)=\lambda_2'(\lambda_1'+2) \\ \lambda_2^2 = (\lambda_2')^2 \end{cases}. \]
which is easily solved.
\end{proof}

Denote by $\Delta^\g(\mu) = U(\g) \tens{U(\bb)} \C_\mu$ the classical Verma module for $\g$ with highest weight $\mu \in \C$, and by $P^\g(\mu)$ its indecomposable projective cover in the category $\mathcal{O}$ for $\g$. Recall that, if $\mu \in \Z_{\geq 0}$, $P^\g(-\mu-2)$ is the unique non-trivial extension of $\Delta^\g(-\mu-2)$ by $\Delta^\g(\mu)$, and that there are no extensions between other $\Delta^\g$'s (inside category $\mathcal{O}$).

\begin{lemma} \label{lemma:structure_Verma}
As a $\g$-module, $\Delta(\lambda)$ has a filtration with subquotients isomorphic to $\Delta^\g(\lambda_1 - 2k)$, $k = 0, 1, 2, \ldots$.

If $\lambda_2 = 0$ or $\lambda_1 \not\in \Z_{\geq 0}$, then, as a $\g$-module, we have
\[ \Delta(\lambda) \cong \bigoplus_{k \geq 0} \Delta^\g(\lambda_1 - 2k). \]
Otherwise (i.e. if $\lambda_2 \neq 0$ and $\lambda_1 \in \Z_{\geq 0}$), we have, 
as $\g$-modules,
\[ \Delta(\lambda) \cong \begin{cases} \displaystyle \bigoplus_{k =1}^{\frac{\lambda_1}{2}+1} P^\g(-2k) \oplus \bigoplus_{k \geq 2} \Delta^\g(-\lambda_1 - 2k) &\colon \lambda_1 \text{ even}, \\
\displaystyle \Delta^\g(-1)\oplus \bigoplus_{k =1}^{\frac{\lambda_1+1}{2}} P^\g(-2k-1) \oplus \bigoplus_{k \geq 2} \Delta^\g(-\lambda_1 - 2k) &\colon \lambda_1 \text{ odd}.  \end{cases}
\]
\end{lemma}
\begin{proof}
Denote by $v_\lambda$ a basis element of $\C_\lambda$. Then $\Delta(\lambda)$ has a basis of weight vectors $\{ f^i \bar{f}^j v_\lambda \colon i,j \geq 0\}$. A direct computation (with help of \cite[Lemma~21.2]{humphreys1978introduction} and its Takiff analogue, alternatively use \cite[Lemma~2.1]{cai2016whittaker}) shows that
\begin{align}
\label{equation:e_Verma} e \cdot f^i \bar{f}^j v_\lambda &= [e , f^i] \bar{f}^j v_\lambda +  f^i [e,\bar{f}^j ] v_\lambda \\
\nonumber &= i(\lambda_1 - i -2j +1) f^{i-1} \bar{f}^j v_\lambda + j \lambda_2 f^i \bar{f}^{j-1} v_\lambda.
\end{align}
This implies that the required filtration is given by the degree of $\bar{f}$. The subquotients are given by the span of $\{ f^i \bar{f}^k v_\lambda \colon i \geq 0\}$, which is clearly isomorphic to $\Delta^\g(\lambda_1 - 2k)$.

If $\lambda_2 =0$, it is clear that the span of $\{ f^i \bar{f}^k v_\lambda \colon i \geq 0\}$, $k$ fixed, is a $\g$-submodule. If $\lambda_1 \not\in \Z_{\geq 0}$, then there are no possible non-trivial extensions between $\Delta^\g(\lambda_1 - 2k)$, $k \geq 0$, hence $\Delta(\lambda)$ splits into as a direct sum of these.

Suppose now $\lambda_2 \neq 0$ and $\lambda_1 \in \Z_{\geq 0}$. Fix $\mu \in \{0,1,\ldots,\lambda_1\}$ of the same parity as $\lambda_1$. It is enough to show that $\Delta^\g(-\mu - 2)$ is not a $\g$-submodule of $\Delta(\lambda)$. Suppose it is. Its highest weight vector $v_{-\mu-2}$ must be a non-trivial linear combination of $f^i \bar{f}^j v_\lambda$ with $i+ j = \frac{\lambda_1+\mu}{2}+1 =:t$, with a non-zero coefficient by $\bar{f}^{t} v_\lambda$.

From (\ref{equation:e_Verma}) it follows that the matrix of $e$ in bases $f^{i} \bar{f}^{t-i} v_\lambda$, $i=0,\ldots,t$, and $f^{i} \bar{f}^{t-1-i} v_\lambda$, $i=0,\ldots,t-1$, has the form
\[ \begin{pmatrix}
\ast & \ast & 0    & \ldots & 0 & 0 & 0      \\
0 & \ast & \ast    & \ldots & 0 & 0 & 0      \\
0 & 0 & \ast       & \ddots & \vdots & \vdots & \vdots \\
\vdots & \vdots & \vdots & \ddots &  \ast & 0 & 0 \\
0 & 0 & 0 & \ldots &  \ast & \ast & 0 \\
0 & 0 & 0 & \ldots &  0 & \ast & \ast
\end{pmatrix}, \]
with all $\ast$ non-zero, except the one on the position $(\mu+1,\mu+2)$, where we have zero (because the bracket in (\ref{equation:e_Verma}) is zero for $f^{\mu+1} \bar{f}^{\left(\frac{\lambda_1-\mu}{2} \right)} v_\lambda$). From this, it follows that $e$ cannot annihilate $v_{-\mu-2}$, a contradiction.
\end{proof}

\begin{lemma} \label{lemma:verma_otimes_fin}
For $\lambda_2 \neq 0$ and $\mu \in \Z_{\geq 0}$ there is an isomorphism of $\T$-modules 
\[ \Delta(\lambda) \tens{} L(\mu) \cong \Delta(\lambda_1+\mu,\lambda_2) \oplus \Delta(\lambda_1+\mu-2,\lambda_2) \oplus \ldots \oplus \Delta(\lambda_1-\mu,\lambda_2). \]
\end{lemma}
\begin{proof}
In the same way as for the semi-simple case, see e.g. \cite[6.3.]{humphreys2008representations} one sees that the left-hand side has a filtration with subquotients equal to the summands on the right-hand side. But these subquotients have different central characters, which follows from Lemma \ref{lemma:Verma_central_ch}, so they split.
\end{proof}

\subsection{Enright-Arkhipov completion}

Here we show that $\g$-Harish-Chandra modules naturally occur in a certain completion (or localization) of Verma modules. We consider a combination of two of such constructions, originally given by Enright in \cite{enright1979on}, and Arkhipov in \cite{arkhipov2004algebraic}. See also \cite{deodhar1980on, andersen2003twisting, konig2002enrights, khomenko2005on}.
To ease the notation a little bit, we will write $U$ instead of $U(\T)$ for the rest of
this section.

Fix an $\ad$-nilpotent element $x \in \T$ (for example $f$, $e$, or $\bar{e}$, which we will use), and denote by $U_{(x)}$ the localization of the algebra $U$ by the multiplicative set generated by $x$. This localization satisfies the Ore conditions by \cite[Lemma 4.2.]{mathieu2000classification}, but this is also visible from the proof of Lemma \ref{lemma:ann_loc}. Since $U$ has no zero-divisors, the canonical map $U \to U_{(x)}$ is injective. Hence we may consider  the $U$-$U$-bimodule
\[ S_x := \quotient{U_{(x)}}{U}. \]

\begin{lemma} \label{lemma:basis_localization}
\begin{enumerate}[(a)] 
\item Suppose $\{x, x_1, \ldots, x_{5}\}$ is a basis for $\T$. The set of all monomials $\{ x^{k} x_1^{k_1} \ldots x_5^{k_5} \colon k \in \Z, \ k_1,\ldots, k_5 \in \Z_{\geq 0} \}$ is a basis for $U_{(x)}$.

\item The analogous set, but with $k\in\Z_{<0}$, is a basis for the quotient $S_x$.
\end{enumerate}

\end{lemma}
\begin{proof}
The set in the first claim is a generating set for $U_{(x)}$, which follows from PBW and the properties of Ore localization. But this set is also linearly independent, since for its any finite subset, the multiplication from the left by $x^m$ for some large $m$ produces a linearly independent set in $U \leq U_{(x)}$. This proves the first claim and
the second claim follows from it.
\end{proof}

Denote by $j \colon M \to U_{(x)} \tens{U} M$ the canonical map. By using the right exactness of the tensor product, we can identify
\begin{equation} \label{equation:canonical_copy_localization}
S_x \tens{U} M \cong \quotient{\displaystyle\left(U_{(x)} \tens{U} M \right)}{j(M)}.
\end{equation}
Moreover, if $M$ is a $\T$-module on which $x$ acts injectively, then the canonical map $j$ is injective. In particular, this is true if $M$ is a Verma module $\Delta(\lambda)$ and $x=f$. 

\begin{lemma}[\cite{deodhar1980on, andersen2003twisting}] \label{lemma:localization_tensor_L}
Fix $x \in \{f,e,\bar{f},\bar{e}\}$, let $M$ be $\T$-module, and $L$ a finite-dimensional $\T$-module. Then there is a natural isomorphism of $\T$-modules
\[ S_x \tens{U} \left(M \otimes L \right) \cong \left(S_x \tens{U} M\right) \tens{} L . \]
\end{lemma}
\begin{proof}
There is an isomorphism $U_{(x)} \otimes_{U} (M \otimes L) \to (U_{(x)} \otimes_{U} M) \otimes L$ given by
\[ x^{-n} \otimes (m \otimes v) \mapsto \sum_{k\geq 0} (-1)^k {n+k-1 \choose k} (x^{-n-k} \otimes m) \otimes x^k v, \]
with the inverse given by $(x^{-n} \otimes m) \otimes v \mapsto x^{-ar} \otimes \sum_{k \geq 0} {ar \choose k}( x^{ar-n-k} m \otimes x^kv)$, where $r,a \in \Z_{>0}$ are chosen so that $x^r$ annihilates $L$ and $(r-1)a \geq n$. This is proved in \cite[Theorem 3.1]{deodhar1980on} and \cite[Theorem 3.2.]{andersen2003twisting} for the semi-simple case, but the proof is analogous in general. In proving that these maps compose to the identity, the following combinatorial formula is helpful: $\sum_{k=0}^n (-1)^k{a \choose n-k} {b+k \choose k} = {a-b-1 \choose n}$.

One can check that these isomorphisms preserve the canonical images of $M \otimes L$ in both sides, see (\ref{equation:canonical_copy_localization}), so they induce the required isomorphisms on the quotients.
\end{proof}

For a $\T$-module $M$, we write $\prescript{x}{}\!M$ for the set of all elements $m \in M$ for which the action of $x$ is locally finite, in the sense that $\dim \C[x]m<\infty$. Note that this is a variant of the Zuckerman functor.
\begin{lemma}
For a $\T$-module $M$, $\prescript{x}{}\!M$ is a $\T$-submodule. Moreover, the assignment $M \mapsto \prescript{x}{}\!M$ is a left-exact functor in the obvious way.
\end{lemma}
\begin{proof}
Since $x$ is assumed to be $\ad$-nilpotent, the claim follows from the formula in \cite[Lemma 21.4.]{humphreys1978introduction}.
\end{proof}

\begin{definition} \label{definition:EA}
For a $\T$-module $M$, define
\begin{equation}
\EnAr(M) := \prescript{e}{}{\left( S_f \tens{U} M \right)}.
\end{equation}
This is a functor on the category of $\T$-modules in the obvious way, which we call \emph{Enright-Arkhipov's completion functor}.
\end{definition}


\begin{proposition}
\label{proposition:EA_tensor_L}
The functor $\EnAr$ commutes with tensoring with a finite-dimensional $\T$-module. More precisely, let $M$ be a $\T$-module, and $L$ a finite-dimensional $\T$-module. Then there is a natural isomorphism of $\T$-modules
\[ \EnAr(M \otimes L) \cong \EnAr(M) \otimes L. \]
\end{proposition}
\begin{proof}
Because of Lemma \ref{lemma:localization_tensor_L}, it is enough to show that $\prescript{e}{}(M \otimes L) = (\prescript{e}{}\!M) \otimes L$ for $\g$-modules $M$ and $L$ with $L=L(\mu)$ simple finite-dimensional. This is proved in \cite[Corollary 3.2]{deodhar1980on}, but we also give a proof for the sake of completeness.

The inclusion $(\prescript{e}{}\!M) \otimes L \subseteq \prescript{e}{}(M \otimes L)$ is trivial. For the converse, denote by $v$ the lowest weight vector of $L$. Then $v, e v, \ldots, e^{\mu} v$ is a basis for $L$. Take a general element $m = \sum_{i=0}^\mu m_i \otimes e^i v \in \prescript{e}{}(M \otimes L)$, and observe that for $n > \mu$ we have
\begin{align*}
e^n \cdot m &= \sum_{i=0}^\mu \sum_{j=0}^{\mu-i} {n \choose j}e^{n-j} m_i \otimes e^{i+j} v \\
&= \sum_{i=0}^\mu \left( e^n m_i + \sum_{j=0}^{i-1} {n \choose i-j}e^{n+j-i} m_j \right) \otimes e^{i} v .
\end{align*}
For a fixed $i$, the vectors inside the big brackets must span a finite-dimensional space when $n$ varies. From this, by an induction on $i$ follows that $e^n m_i$ span a finite-dimensional space, hence $m \in (\prescript{e}{}\!M) \otimes L$.
\end{proof}

\begin{example} \label{example:EA_Delta_g}
Let us consider $\Delta^\g(\mu)$, with $\mu \in \C$. From Lemma \ref{lemma:basis_localization} it follows that the set $\{ f^{-k} v_\mu \colon k>0\}$ is a basis for $S_f \tens{U} M$ (and from an argument for linear independence very similar to the one in the proof of Lemma \ref{lemma:basis_localization}). One can easily prove by induction the following commutation relations (similar to \cite[3.5]{mazorchuk2010lectures}):
\begin{align} \label{equation:commutators_he_f_inv}
&[h,f^{-k}] = 2k f^{-k}, \\
\nonumber & [e,f^{-k}] = -k f^{-k-1}(h+k+1).
\end{align}
From this, it is not hard to see that
\[ \EnAr(\Delta^\g(\mu)) \cong \begin{cases} L(-\mu-2) &\colon \mu \in \Z \text{ and } \mu \leq -2, \\ 0 &\colon \text{otherwise}. \end{cases} \]
Similarly, one sees that $\EnAr(P^\g(\mu))=0$ for $\mu \in \Z$ and $\mu \leq -2$. (Or using the fact that big projective modules can be obtained by tensoring dominant Verma modules with finite-dimensional modules, together with Proposition \ref{proposition:EA_tensor_L}).
\end{example}

Recall that we use notation $\lambda=(\lambda_1,\lambda_2) \in \tilde{\h}^\ast$, with $\lambda_1 = \lambda(h)$ and $\lambda_2 = \lambda(\bar{h})$.

\begin{theorem} \label{theorem:EA_Verma}
Take $\lambda$ with $\lambda_1 \in \Z$ and $\lambda_2 \neq 0$. Then $\EnAr(\Delta(\lambda))$ is a simple $\g$-Harish-Chandra module. As a $\g$-module, it decomposes as follows:
\begin{equation} \label{equation:g_types_EA}
\EnAr(\Delta(\lambda)) \cong \bigoplus_{k \geq 0} L(|\lambda_1 + 2| +2k).
\end{equation}
\end{theorem}
\begin{proof}
Lemma \ref{lemma:structure_Verma}, the fact that the functor $\EnAr$ commutes with the forgetful functor from $\T$-modules to $\g$-modules, and Example \ref{example:EA_Delta_g} imply (\ref{equation:g_types_EA}).

From Lemma \ref{lemma:basis_localization} we have a basis for $S_f \tens{U} \Delta(\lambda)$ consisting of $f^{-i} \bar{f}^j v_\lambda$, for $i \geq 1$ and $j \geq 0$. Since the lowest weight vector of a $\g$-type $L(\mu)$ ($\mu$ of the same parity as $\lambda_1$) inside $\EnAr(\Delta(\lambda))$ must be annihilated by $f$, it must be (up to scalar) equal to $f^{-1} \bar{f}^t v_\lambda$, where $t = \frac{\mu+\lambda_1}{2}+1$.

Now we will prove that $\EnAr(\Delta(\lambda))$ is simple. Let $V$ be its non-zero submodule, and suppose it contains $L(\mu)$ for some $\mu$ from (\ref{equation:g_types_EA}). By applying $\bar{f}$ on $f^{-1} \bar{f}^t v_\lambda$, we get that $L(\mu+2k) \subseteq V$ for all $k \geq 0$.

To prove that $V=\EnAr(\Delta(\lambda))$, it is enough to assume $\mu > |\lambda_1+2|$ and to find an element in $U(\T)$ that maps $f^{-1} \bar{f}^t v_\lambda$ to $f^{-1} \bar{f}^{t-1} v_\lambda$.

In addition to (\ref{equation:commutators_he_f_inv}), we will use the following commutation relations, whose proofs are analogous to the ones for (\ref{equation:commutators_he_f_inv}):
\begin{align}
&[\bar{h},f^{-k}] = 2k f^{-k-1} \bar{f}, \\
\nonumber &[\bar{e},f^{-k}] = - kf^{-k-1} \bar{h} - k(k+1)f^{-k-2} \bar{f}.
\end{align}
From this, we have:
\begin{align*}
 e \cdot f^{-1} \bar{f}^t v_\lambda &= [e,f^{-1}] \bar{f}^t v_\lambda + f^{-1} [e,\bar{f}^t] v_\lambda \\
&=  - f^{-2} (h+2) \bar{f}^t v_\lambda + t \lambda_2 f^{-1} \bar{f}^{t-1} v_\lambda \\
&= \mu f^{-2} \bar{f}^t v_\lambda + t \lambda_2 f^{-1} \bar{f}^{t-1} v_\lambda, \\[.5em]
\bar{h}e \cdot f^{-1} \bar{f}^t v_\lambda &=  \mu \bar{h} f^{-2} \bar{f}^t v_\lambda + t \lambda_2 \bar{h} f^{-1} \bar{f}^{t-1} v_\lambda \\
&= \mu [\bar{h}, f^{-2}] \bar{f}^t v_\lambda + \mu \lambda_2 f^{-2} \bar{f}^t v_\lambda  + t\lambda_2 [\bar{h},f^{-1}] \bar{f}^{t-1} v_\lambda + t\lambda_2^2 f^{-1} \bar{f}^{t-1} v_\lambda \\
 &= 4\mu f^{-3} \bar{f}^{t+1} v_\lambda + (\lambda_1 + 2\mu + 2)\lambda_2 f^{-2} \bar{f}^{t} v_\lambda + t\lambda_2^2 f^{-1} \bar{f}^{t-1} v_\lambda, \\[.5em]
2\mu \bar{e} \cdot f^{-1} \bar{f}^t v_\lambda &= 2\mu[\bar{e}, f^{-1}] \bar{f}^t v_\lambda \\
 &=   -4\mu f^{-3} \bar{f}^{t+1} v_\lambda - 2\mu\lambda_2 f^{-2} \bar{f}^t v_\lambda.
\end{align*}
From this it follows that
\[ (\bar{h}e - 2\mu \bar{e}) \cdot f^{-1} \bar{f}^t v_\lambda = (\lambda_1 + 2)\lambda_2 f^{-2} \bar{f}^{t} v_\lambda + t\lambda_2^2 f^{-1} \bar{f}^{t-1} v_\lambda. \]
Now we claim that a non-trivial linear combination of $e$ and $(\bar{h}e - 2\mu \bar{e})$ will map $f^{-1} \bar{f}^t v_\lambda$ to $f^{-1} \bar{f}^{t-1} v_\lambda$. This is true, because the determinant
\[ \begin{vmatrix} \mu & (\lambda_1+2)\lambda_2 \\ t\lambda_2 & t\lambda_2^2 \end{vmatrix} = \lambda_2^2(\mu+\lambda_1+2)(\mu-\lambda_1-2) \neq 0. \]
This finishes the proof of simplicity.
\end{proof}

\begin{remark}[A sketch of an alternative proof of simplicity of $\EnAr(\Delta(\lambda))$ for $\lambda_1 \in \Z$ and $\lambda_2\neq 0$] \label{remark:simpllicity_EA}
Suppose $V$ is a submodule of $\EnAr(\Delta(\lambda))$ having $L(\mu)$, $\mu >|\lambda_1+2|$ minimal. By applying $\bar{f}$, we see that as a $\g$-module, $V \cong L(\mu) \oplus L(\mu+2) \oplus \ldots$. This implies that the quotient ${\EnAr(\Delta(\lambda))}/{V} \cong L(|\lambda_1+2|) \oplus L(|\lambda_1+2|+2) \oplus \ldots \oplus L(\mu-2)$
is simple as a $\T$-module and finite-dimensional. Since $\bar{C}$ consists of elements from the nilradical of $\T$, it must act as zero on this quotient. But $\bar{C}$ is central, so it still acts as $\lambda^2_2$ on the localization, a contradiction.
\end{remark}

\subsection{Classification}

In this subsection, we use the relation with highest weight theory established above to classify all simple $\g$-Harish-Chandra modules for $\T$. It will be more convenient 
to shift the notation for the first parameter in our modules by $-2$.

\begin{definition}
For $n \in \Z$ and $\lambda_2 \neq 0$, denote
\begin{equation*}
V(n,\lambda_2) := \EnAr(\Delta(n-2, \lambda_2)).
\end{equation*} 
\end{definition}

\begin{corollary} \label{corollary:g_types}
The module $V(n,\lambda_2)$ is a simple $\g$-Harish-Chandra module, it
has $\g$-types $L(|n|), L(|n|+2), L(|n|+4) \ldots$ and each of these
occurs with multiplicity one.

If $V(n,\lambda_2) \cong V(n',\lambda_2')$, then $(n',\lambda_2') = (n,\lambda_2)$ or $(-n,-\lambda_2)$.
\end{corollary}
\begin{proof}
The first statement follows from Theorem \ref{theorem:EA_Verma}.

It is clear that the functor $\EnAr$ preserves central character. So, the generators of the center $C$ and $\bar{C}$ act as the scalars $n\lambda_2$ and $\lambda_2^2$, respectively. From this the second statement follows.
\end{proof}

We will see later in this subsection that the modules $V(n,\lambda_2)$ exhaust all infinite-dimensional simple $\g$-Harish-Chandra modules.

On Figure \ref{figure:V_1} to \ref{figure:V2} we present several $V(n,\lambda_2)$'s, and how they are constructed. The gray area on the left hand-side is the Verma module, decomposed into rows according to Lemma \ref{lemma:structure_Verma}, and furthermore, into weight spaces. The remaining bullets represent $S_f$ tensored with the Verma module. The arrows represent non-zero action of $e$, and the light-gray area on the right hand-side contains vectors not having a finite $e$-orbit. The remaining (not shaded) part is our $V(n,\lambda_2)$, with its $\g$-types clearly visible.
\begin{figure}[h]
%

%
%
\begin{tikzpicture}[scale=.7] 
\def\epsilon{.2}

\newcommand{\Verma}[6]{
	\draw [fill=gray!30] (#1-\epsilon,#2+\epsilon) --(#3+2*\epsilon,#4+\epsilon) -- (#5-\epsilon,#6-2*\epsilon);
	}

\newcommand{\arrow}[4]{
	\draw [->] (#1-\epsilon*#1+\epsilon*#3,#2-\epsilon*#2+\epsilon*#4) -- (\epsilon*#1+#3-\epsilon*#3,\epsilon*#2+#4-\epsilon*#4);
	}

\Verma{-4}{0}{-2}{0}{-4}{-2}

\path [fill=gray!15] (3+\epsilon,0+\epsilon) --(1-2*\epsilon,0+\epsilon) -- (3+\epsilon,-2-2*\epsilon);

\foreach \x in {-4,-3,-2,-1,0,1,2,3}{
	\foreach \y in {0,-1,-2}{ 
    	\node[draw,circle,inner sep=1pt,fill] at (\x,\y) {};
  		}
	}

\foreach \y in {0,-1,-2}{ 
	\node[] at (-5,\y) {$\ldots$};
	\node[] at (4,\y) {$\ldots$};
	}

\node[] at (-.5,-2.75) {$\ldots$};

\foreach \x in {-7,-5,-3,-1,1,3,5,7}{
	\node[above,yshift=5] at (\x/2-1/2,0) {{\scriptsize $\x$}};
	}

\arrow{-5}{0}{-4}{0}\arrow{-4}{0}{-3}{0}\arrow{-3}{0}{-2}{0};\arrow{-1}{0}{0}{0};\arrow{1}{0}{2}{0};\arrow{2}{0}{3}{0};\arrow{3}{0}{4}{0};

\arrow{-5}{-1}{-4}{-1}\arrow{-4}{-1}{-3}{-1}\arrow{-2}{-1}{-1}{-1};\arrow{-1}{-1}{0}{-1};\arrow{0}{-1}{1}{-1};\arrow{2}{-1}{3}{-1};\arrow{3}{-1}{4}{-1};

\arrow{-5}{-2}{-4}{-2}\arrow{-3}{-2}{-2}{-2}\arrow{-2}{-2}{-1}{-2};\arrow{-1}{-2}{0}{-2};\arrow{0}{-2}{1}{-2};\arrow{1}{-2}{2}{-2};\arrow{3}{-2}{4}{-2};
\end{tikzpicture}
\caption{$V(-1,\lambda_2) = \EnAr(\Delta(-3,\lambda_2))$}
\label{figure:V_1}
\bigskip\bigskip
\begin{tikzpicture}[scale=.7] 
\def\epsilon{.2}

\newcommand{\Verma}[6]{
	\draw [fill=gray!30] (#1-\epsilon,#2+\epsilon) --(#3+2*\epsilon,#4+\epsilon) -- (#5-\epsilon,#6-2*\epsilon);
	}

\newcommand{\arrow}[4]{
	\draw [->] (#1-\epsilon*#1+\epsilon*#3,#2-\epsilon*#2+\epsilon*#4) -- (\epsilon*#1+#3-\epsilon*#3,\epsilon*#2+#4-\epsilon*#4);
	}

\Verma{-3}{0}{-1}{0}{-3}{-2}

\path [fill=gray!15] (3+\epsilon,0+\epsilon) --(1-2*\epsilon,0+\epsilon) -- (3+\epsilon,-2-2*\epsilon);

\foreach \x in {-3,-2,-1,0,1,2,3}{
	\foreach \y in {0,-1,-2}{ 
    	\node[draw,circle,inner sep=1pt,fill] at (\x,\y) {};
  		}
	}

\foreach \y in {0,-1,-2}{ 
	\node[] at (-4,\y) {$\ldots$};
	\node[] at (4,\y) {$\ldots$};
	}

\node[] at (0,-2.75) {$\ldots$};

\foreach \x in {-6,-4,-2,0,2,4,6}{
	\node[above,yshift=5] at (\x/2,0) {{\scriptsize $\x$}};
	}

\arrow{-4}{0}{-3}{0};\arrow{-3}{0}{-2}{0};\arrow{-2}{0}{-1}{0};\arrow{1}{0}{2}{0};\arrow{2}{0}{3}{0};\arrow{3}{0}{4}{0};

\arrow{-4}{-1}{-3}{-1};\arrow{-3}{-1}{-2}{-1};\arrow{-1}{-1}{0}{-1};\arrow{0}{-1}{1}{-1};\arrow{2}{-1}{3}{-1};\arrow{3}{-1}{4}{-1};

\arrow{-4}{-2}{-3}{-2};\arrow{-2}{-2}{-1}{-2};\arrow{-1}{-2}{0}{-2};\arrow{0}{-2}{1}{-2};\arrow{1}{-2}{2}{-2};\arrow{3}{-2}{4}{-2};
\end{tikzpicture}
\caption{$V(0,\lambda_2) = \EnAr(\Delta(-2,\lambda_2))$}
\label{figure:V0}
\bigskip\bigskip
\begin{tikzpicture}[scale=.7] 
\def\epsilon{.2}

\newcommand{\Verma}[6]{
	\draw [fill=gray!30] (#1-\epsilon,#2+\epsilon) --(#3+2*\epsilon,#4+\epsilon) -- (#5-\epsilon,#6-2*\epsilon);
	}

\newcommand{\arrow}[4]{
	\draw [->] (#1-\epsilon*#1+\epsilon*#3,#2-\epsilon*#2+\epsilon*#4) -- (\epsilon*#1+#3-\epsilon*#3,\epsilon*#2+#4-\epsilon*#4);
	}

\Verma{-3}{0}{-1}{0}{-3}{-2}

\path [fill=gray!15] (2+\epsilon,0+\epsilon) --(0-2*\epsilon,0+\epsilon) -- (2+\epsilon,-2-2*\epsilon);

\foreach \x in {-3,-2,-1,0,1,2}{
	\foreach \y in {0,-1,-2}{ 
    	\node[draw,circle,inner sep=1pt,fill] at (\x,\y) {};
  		}
	}

\foreach \y in {0,-1,-2}{ 
	\node[] at (-4,\y) {$\ldots$};
	\node[] at (3,\y) {$\ldots$};
	}

\node[] at (-.5,-2.75) {$\ldots$};

\foreach \x in {-5,-3,-1,1,3,5}{
	\node[above,yshift=5] at (\x/2-1/2,0) {{\scriptsize $\x$}};
	}

\arrow{-4}{0}{-3}{0};\arrow{-3}{0}{-2}{0};\arrow{-2}{0}{-1}{0};\arrow{0}{0}{1}{0};\arrow{1}{0}{2}{0};\arrow{2}{0}{3}{0};

\arrow{-4}{-1}{-3}{-1};\arrow{-3}{-1}{-2}{-1};\arrow{-1}{-1}{0}{-1};\arrow{1}{-1}{2}{-1};\arrow{2}{-1}{3}{-1};

\arrow{-4}{-2}{-3}{-2};\arrow{-2}{-2}{-1}{-2};\arrow{-1}{-2}{0}{-2};\arrow{0}{-2}{1}{-2};\arrow{2}{-2}{3}{-2};
\end{tikzpicture}
\caption{$V(1,\lambda_2) = \EnAr(\Delta(-1,\lambda_2))$}
\label{figure:V1}
\bigskip\bigskip
\begin{tikzpicture}[scale=.7] 
\def\epsilon{.2}

\newcommand{\Verma}[6]{
	\draw [fill=gray!30] (#1-\epsilon,#2+\epsilon) --(#3+2*\epsilon,#4+\epsilon) -- (#5-\epsilon,#6-2*\epsilon);
	}

\newcommand{\arrow}[4]{
	\draw [->] (#1-\epsilon*#1+\epsilon*#3,#2-\epsilon*#2+\epsilon*#4) -- (\epsilon*#1+#3-\epsilon*#3,\epsilon*#2+#4-\epsilon*#4);
	}

\Verma{-3}{0}{0}{0}{-3}{-3}

\path [fill=gray!15] (3+\epsilon,0+\epsilon) -- (1-\epsilon,0+\epsilon) -- (0-3*\epsilon,-1-\epsilon) -- (1-\epsilon,-1-\epsilon) -- (3+\epsilon,-3-3*\epsilon);

\foreach \x in {-3,-2,-1,0,1,2,3}{
	\foreach \y in {0,-1,-2,-3}{ 
    	\node[draw,circle,inner sep=1pt,fill] at (\x,\y) {};
  		}
	}

\foreach \y in {0,-1,-2,-3}{ 
	\node[] at (-4,\y) {$\ldots$};
	\node[] at (4,\y) {$\ldots$};
	}

\node[] at (0,-3.75) {$\ldots$};

\foreach \x in {-6,-4,-2,0,2,4,6}{
	\node[above,yshift=5] at (\x/2,0) {{\scriptsize $\x$}};
	}

\arrow{-4}{0}{-3}{0};\arrow{-3}{0}{-2}{0};\arrow{-2}{0}{-1}{0};\arrow{1}{0}{2}{0};\arrow{2}{0}{3}{0};\arrow{3}{0}{4}{0};

\arrow{-4}{-1}{-3}{-1};\arrow{-3}{-1}{-2}{-1};\arrow{-2}{-1}{-1}{-1};\arrow{1}{-1}{2}{-1};\arrow{2}{-1}{3}{-1};\arrow{3}{-1}{4}{-1};

\arrow{-4}{-2}{-3}{-2};\arrow{-3}{-2}{-2}{-2};\arrow{-1}{-2}{0}{-2};\arrow{0}{-2}{1}{-2};\arrow{2}{-2}{3}{-2};\arrow{3}{-2}{4}{-2};

\arrow{-4}{-3}{-3}{-3};\arrow{-2}{-3}{-1}{-3};\arrow{-1}{-3}{0}{-3};\arrow{0}{-3}{1}{-3};\arrow{1}{-3}{2}{-3};\arrow{3}{-3}{4}{-3};

\arrow{-1}{-1}{0}{0};\arrow{0}{-1}{1}{0};\arrow{1}{-1}{2}{0};\arrow{2}{-1}{3}{0};\arrow{3}{-1}{4}{0};
\end{tikzpicture}
\caption{$V(2,\lambda_2) = \EnAr(\Delta(0,\lambda_2))$}
\label{figure:V2}
\end{figure}

From Corollary \ref{corollary:uniqueness_with_0},  Theorem \ref{theorem:Q_simple} and Corollary \ref{corollary:g_types}
we have the following consequence:
\begin{corollary} \label{corollary:Q0_V0}
For $\lambda_2 \neq 0$ we have $V(0,\lambda_2) \cong V(0,-\lambda_2) \cong Q(0,\lambda_2^2)$.
\end{corollary}

From Lemma \ref{lemma:verma_otimes_fin}, Proposition \ref{proposition:EA_tensor_L}, and the definition of $V(n,\lambda_2)$, we have:
\begin{proposition} \label{proposition:V_otimes_L}
For $n \in \Z$, $\lambda_2 \neq 0$ and $\mu \in \Z_{\geq 0}$, we have the following isomorphism of $\T$-modules:
\[ V(n,\lambda_2) \tens{} L(\mu) \cong V(n+\mu,\lambda_2) \oplus V(n+\mu-2,\lambda_2) \oplus \ldots \oplus V(n-\mu,\lambda_2). \]
\end{proposition}

Now we can completely describe the universal modules:

\begin{proposition} \label{proposition:Q_n}
For $n \in \Z_{\geq 0}$ and $\chi \neq 0$, choose any square root $\lambda_2$ of $\chi$. Then
\[ Q(n,\chi) \cong V(n,\lambda_2) \oplus V(n-2,\lambda_2) \oplus \ldots \oplus V(-n,\lambda_2). \]
Moreover, $V(n,\lambda_2) \cong V(-n,-\lambda_2)$.
\end{proposition}
\begin{proof}
The first claim follows from Proposition \ref{proposition:V_otimes_L}, Corollary \ref{corollary:Q0_V0} and (\ref{equation:tensor2}). The second claim follows from the first one by comparing both choices $\pm \lambda_2$ and central characters of the summands.
\end{proof}

\begin{theorem} \label{theorem:classification}
Let $V$ be a simple $\g$-Harish-Chandra module for $\T$. Denote by $\chi = \chi(\bar{C})$ the purely Takiff part of the central character, and suppose $L(n)$, $n \in \Z_{\geq 0}$, is the minimal $\g$-type of $V$.
\begin{itemize}
\item If $\chi \neq 0$, then $V \cong V(n,\lambda_2)$, for a square root $\lambda_2$ of $\chi$.
\item If $\chi = 0$, then $V \cong L(n)$ with the trivial $\bar{\g}$-action.
\end{itemize}
In other words, $V(n,\lambda_2)$, $n \in \Z$, $\lambda_2 \in \C\setminus\{0\}$, together with the finite-dimensional simple $\g$-modules constitute a complete list of simple $\g$-Harish-Chandra modules for $\T$. The only isomorphisms between different members of the list are $V(n,\lambda_2) \cong V(-n,-\lambda_2)$. 
\end{theorem}
\begin{proof}
By Proposition \ref{proposition:multiplicity}(\ref{item:universal_property}), $V$ is a quotient of $Q(n,\chi)$.

If $\chi\neq 0$, from Proposition \ref{proposition:Q_n} and Corollary \ref{corollary:g_types} we see that the only possible choices with the correct minimal $\g$-type are $V(n,\lambda_2)$ or $V(n,-\lambda_2)$.

If $\chi = 0$, by the second part of Theorem \ref{theorem:Q_simple}, we see that the only possible simple quotients of $Q(n,0) \cong Q(0,0) \otimes L(n)$ are just finite-dimensional simple $\g$-modules with the trivial $\bar{\g}$-action.
\end{proof}

\subsection{Extensions} \label{subsection:extensions}

Here we calculate the first extension groups of simple $\g$-Harish-Chandra modules, restricting to the infinite-dimensional cases, i.e., a non-trivial central character. Since in that case non-isomorphic $\g$-Harish-Chandra modules have different central characters, there are no non-trivial extensions between them. So it only makes sense to calculate the self-extensions. 

\begin{theorem} \label{theorem:extensions}
For an infinite-dimensional simple $\g$-Harish-Chandra module $V$, we have \[ \Ext^1(V,V) \cong \C. \]
\end{theorem}
\begin{proof}
Assume first that $V= Q(0,\chi)$ for $\chi \neq 0$, and suppose we have a non-split short exact sequence $0 \to V \stackrel{i}{\hookrightarrow} M \stackrel{p}{\tto} V \to 0$. Denote by $1 \in V$ the generator from $L(0)$, set $w=i(1) \in M$ and find $v \in M$ such that $p(v)=1$. Since the sequence must split in the category of $\g$-modules, $v$ generates the trivial $\g$-submodule in $M$. By the universal property, there is a $\T$-homomorphism $f \colon Q(0) \to M$, and the triangle below commutes:
\[ \xymatrix{ & & Q(0) \ar[d]_{f} \ar[rd] & &\\
0 \ar[r] & V \ar@{^{(}->}[r] & M \ar@{->>}[r] & V \ar[r] & 0.} \]

The map $f$ must be surjective, since otherwise its image would define a splitting of the short exact sequence. So, there is an element in $Q(0)$ that maps to $w$ via $f$; by Lemma \ref{lemma:structure_Q0} and Proposition \ref{proposition:end_Q0} such an element is necessarily of the form $p(\bar{C})$ for some polynomial $p$. Since the triangle above commutes, we must have $p(\chi)=0$. Since $[M \colon L(0)] =2$, we can take $p(\bar{C})=\bar{C}-\chi$. From this one can see that $\Ker f$ is generated by $(\bar{C} - \chi)^2$, i.e., $M \cong \quotient{U(\bar{\g})}{(\bar{C}-\chi)^2}$. This uniquely determines $M$. Conversely, one sees directly that such $M$ defines a non-split self-extension of $Q(0,\chi)$.

The general statement is obtained from this by translation functors, i.e.,  tensoring extensions of $Q(0,\chi)$ by $L(n)$ and then taking the component with the correct central character (see Proposition \ref{proposition:V_otimes_L}). This functor defines a homomorphism of abelian groups $\Ext^1(V(0,\lambda_2),V(0,\lambda_2)) \to \Ext^1(V(n,\lambda_2),V(n,\lambda_2))$. In the same way we get a homomorphism in the other direction.

The fact that these homomorphisms compose to the identities on the $\Ext^1$ groups is an easy application of the $5$-lemma.
\end{proof}

\subsection{Annihilators}

We will prove here that the infinite-dimensional simple $\g$-Harish-Chandra modules have the same annihilators in $U=U(\T)$ as the corresponding Verma modules. We start by showing that, in the cases we are interested in, the localization does not decrease the annihilator. Then we construct a certain ``inverse'' of the functor $\EnAr$, which will produce Verma modules out of $\g$-Harish-Chandra modules. This will be given by the localization by $\bar{e}$, i.e. tensoring with $S_{\bar{e}}$ over $U$. 

\begin{lemma} \label{lemma:ann_loc}
Let $x$ be an $\ad$-nilpotent element in $\T$, and $M$ a $\T$-module on which $x$ acts injectively. Then in $U$ we have
\[ \Ann(M) = \Ann \left(U_{(x)} \tens{U} M \right) \subseteq \Ann \left(S_x \tens{U} M \right). \]
\end{lemma}
\begin{proof} The only non-obvious thing to prove is if $u \in \Ann(M)$, then $u x^{-n} \otimes m =0$, for all $n\geq 1$ and $m\in M$. 

By assumption, for any $u \in U$ there exists $k_1 >0$ such that
\[ 0 = \left(\ad(x)\right)^{k_1}(u) = \sum_{i=0}^{k_1} (-1)^{i}{k_1 \choose i} x^i u x^{k_1-i}, \]
so $x^{k_1} u = u_1 x$ for $u_1 := \sum_{i=0}^{k_1-1} (-1)^{k_1+i}{k_1 \choose i} x^i u x^{k_1-i-1}$. If $u \in \Ann(M)$, then so is $u_1$  too, since $\Ann(M)$ is a two-sided ideal. We can inductively apply the same procedure on $u_1$ to get $k_2$ such that $x^{k_2} u_1 = u_2 x$, etc. Repeating this $n$ times, we get $x^{k_{n}} u_{n-1} = u_n$ for some $u_n \in \Ann(M)$. From the construction it follows that
\[ x^{k_n+\ldots+k_1} \cdot ux^{-n} \otimes m = u_n \otimes m = 1 \otimes u_n m =0. \]
Since $x$ acts injectively on $M$, the same is true for $U_{(x)} \tens{U} M$, so we conclude that $ux^{-n} \otimes m =0$.
\end{proof}

\begin{proposition} \label{proposition:V_localized}
Suppose $n \in \Z$ and $\lambda_2 \neq 0$.
\begin{enumerate}[(a)] 
\item The element $\bar{e}$ acts injectively on $V(n,\lambda_2)$.
\item \label{item:S_bar_e_otimes_V}
The module $S_{\bar{e}} \tens{U} V(n,\lambda_2)$ is isomorphic to the direct sum of Verma modules $\Delta(n-2,\lambda_2) \oplus \Delta(-n-2,-\lambda_2)$.
\end{enumerate}
\end{proposition}
\begin{proof}
The first claim follows from Lemma \ref{lemma:structure_Q0}(\ref{item:basis}) for $V(0,\lambda_2) \cong Q(0,\lambda_2^2)$, and from Proposition \ref{proposition:Q_n} for general $V(n,\lambda_2)$.

The second claim we also prove first for $n=0$, and again translate the result to the other cases. From Lemma \ref{lemma:structure_Q0}(\ref{item:basis}), we get a basis for $W:=S_{\bar{e}} \tens{U} Q(0,\lambda_2^2)$ consisting of
\[ (\bar{e})^{-k} \bar{f}^l \bar{h}^\epsilon, \text{ for } k>0, \ l \geq 0, \ \epsilon \in \{0,1\},  \]
where $\g$ acts by the adjoint action, and $\bar{\g}$ by the (commutative) multiplication. We denote this action of $\T$ by $\circ$.

Consider the following two elements in $W$:
\[ w_\pm := (\bar{e})^{-1} \pm \frac{1}{\lambda_2}(\bar{e})^{-1}\bar{h}.\]
It is an easy calculation to see that $e \circ w_{\pm} = \bar{e} \circ w_{\pm} = 0$, $h \circ w_{\pm} = -2 w_\pm$, and $\bar{h} \circ w_\pm = \pm \lambda_2 w_{\pm}$, from which it follows by the universal property of Verma modules that each $w_\pm$ generates a copy of $\Delta(-2,\pm\lambda_2)$ in $W$. Because of their simplicity, these submodules can only intersect trivially. By comparing the dimensions of weight spaces, we conclude that $W$ cannot have any other composition factor, i.e.,
\begin{equation} \label{equation:W}
W \cong \Delta(-2,\lambda_2) \oplus \Delta(-2,-\lambda_2).
\end{equation}

In general, we calculate $S_{\bar{e}} \tens{U} Q(n,\lambda_2^2)$ in two ways and compare the results:
\begingroup\addtolength{\jot}{.3em}
\begin{align}
\nonumber S_{\bar{e}} {}\tens{U}{} & Q(n,\lambda_2^2) \cong S_{\bar{e}} \tens{U} \left( Q(0,\lambda_2^2) \tens{}  L(n) \right) && \text{by (\ref{equation:tensor2})}, \\
\nonumber\cong & \left( S_{\bar{e}} \tens{U}  Q(0,\lambda_2^2) \right) \tens{}  L(n)  && \text{by Lemma \ref{lemma:localization_tensor_L}}, \\
\nonumber\cong & \, \big( \Delta(-2,\lambda_2) \oplus \Delta(-2,-\lambda_2) \big) \tens{}  L(n)  && \text{by (\ref{equation:W})}, \\
\label{equation:locV_direct_sum1} \cong & \bigoplus_{k=0}^n \Delta(n-2-2k,\lambda_2) \oplus \bigoplus_{k=0}^n \Delta(n-2-2k,-\lambda_2) && \text{by Lemma \ref{lemma:verma_otimes_fin}}.
\end{align}
\endgroup

On the other hand, by Proposition \ref{proposition:Q_n} we have%
\begin{equation} \label{equation:locV_direct_sum2}
S_{\bar{e}} \tens{U} Q(n,\lambda_2^2) \cong \bigoplus_{k=0}^n S_{\bar{e}} \tens{U} V(n-2k,\lambda_2).
\end{equation}

By comparing the central characters (which are preserved under the localization) of the direct summands in (\ref{equation:locV_direct_sum1}) and (\ref{equation:locV_direct_sum2}), the claim (\ref{item:S_bar_e_otimes_V}) follows.
\end{proof}

From Lemma \ref{lemma:ann_loc}, Proposition \ref{proposition:V_localized} and the definition of $V(n,\lambda_2)$ we have:

\begin{corollary} \label{corollary:annihilators}
Suppose $n \in \Z$ and $\lambda_2 \neq 0$. Then 
\[ \Ann(V(n,\lambda_2)) = \Ann(\Delta(n-2,\lambda_2)) = \Ann(\Delta(-n-2,-\lambda_2)). \]
\end{corollary}

We want to prove that these annihilators are centrally generated. It is easier to do this for Verma modules. This has already been proved in \cite[Proposition 6.1]{bavula2017prime}. We present a different and a more direct proof, and along the way reveal some structure of the quotients of $U$ by the centrally generated ideals.

For this, we need to express elements of $U$ modulo a maximal ideal in the center in a convenient way. This we describe in the next two lemmas. We denote by $U_0 := U(\T)_0$, the zero-weight space of $\h$ in $U$.

\begin{lemma} \label{lemma:U0_generators}
The subalgebra $U_0$ of $U$ is generated by $S=\{h, \bar{h}, fe, \bar{f}e, f\bar{e}, \bar{f}\bar{e}\}$.
\end{lemma}
\begin{proof}
We need to prove that any product $x = x_1 x_2 \ldots x_k$, where each $x_i$ belongs to the standard basis of $\T$, with the property that the number of $i$'s for which $x_i \in \{f,\bar{f}\}$ is equal to the number of $j$'s for which $x_j \in \{e,\bar{e}\}$, can be generated by elements in $S$. We prove this by induction on $k$. If $x_1 x_2 \ldots x_k$ consists only of $h$ and $\bar{h}$, we are done. If not, chose some $x_i \in \{f,\bar{f}\}$ and $x_j \in \{e,\bar{e}\}$, and assume without loss of generality $i<j$. We commute them to the right-most place:
\[ x = (\underbrace{x_1 \ldots \hat{x_i} \ldots \hat{x_j} \ldots x_k}_{x'})(\underbrace{x_i x_j}_{\in S}) + \sum_t y_t,\]
where the factors with hat are omitted. It is clear from the commutation relations that $x'$ and all $y_t$ are products of the basis elements with the same property, but shorter. We are done by induction.
\end{proof}

\begin{lemma} \label{lemma:U0quotient_structure}
Fix an algebra homomorphism $\chi \colon Z(\T) \to \C$. For any $u \in \quotient{U_0}{U_0 \cdot \Ker \chi}$ there exists $n \in \Z_{\geq 0}$ such that $(\bar{f}\bar{e})^n \cdot u$ is equal to a linear combination of monomials of the form $h^k \left(f \bar{e}\right)^l \left(\bar{h}\right)^m$ for $k,l,m \in \Z_{\geq 0}$,  modulo $U_0 \cdot \Ker \chi$.
\end{lemma}
\begin{proof}
In the quotient above, by using (\ref{equation:Z_T}) we can express $\bar{f}\bar{e}$ as a linear combination of $\bar{h}^2$ and $1$, and also $\bar{f}e$ as a linear combination of $\bar{h}h$, $\bar{h}$, $f\bar{e}$ and $1$. Using this with Lemma \ref{lemma:structure_Q0}, we see that the generators in the quotient are just $h$, $\bar{h}$, $fe$ and $f\bar{e}$.

First let us assume that $u=x_1 x_2 \ldots x_r$, a product of these four generators in any order. Since $h$ commutes with everything here, we can ignore it. Denote by $a=fe$, $b=f\bar{e}$, $c=\bar{h}$, $\chi_1 = \chi(h)$ and $\chi_2=\chi(\bar{h})$.  One can check that we have the following relations in the quotient:
\begin{align}
\label{equation:b_a} [b,a] &= ac -hb, \\
\label{equation:c_a} [c,a] &= 4b + hc + 2c - \chi_1, \\
\label{equation:c_b} [c,b] &= \frac{1}{2}c^2  -\frac{1}{2}\chi_2, \\
\label{equation:barf_bare_a} \bar{f}\bar{e} \cdot a &= -b^2 - hc^2 - \frac{3}{2}c^2 - hbc - 2bc + \frac{\chi_1}{2}b + \frac{\chi_1}{2}c + \frac{\chi_2}{2}.
\end{align}

Suppose that $x_1 \neq a$, but some $x_i =a$, and assume $i$ is minimal. Using the relations (\ref{equation:b_a}) and (\ref{equation:c_a}), we commute $x_{i-1}x_i = x_i x_{i-1} + [x_{i-1},x_i]$. This way $u$ becomes a sum of several monomials, each of which has either one $a$ less, or have their most-left $a$ one place closer to the most left position. It follows that we can move this $a$ to the most left part in a finite number of steps, i.e., we can express $u$ as a finite sum $u =\sum a y_t + \sum w_t$, where $y_t$ is a finite product of $a$'s, $b$'s and $c$'s, but has at least one $a$ less than the original expression of $u$ had, and $w_t$ is a finite product of $b$'s and $c$'s.

From (\ref{equation:barf_bare_a}) and the relation $4\bar{f}\bar{e} = \chi_2 - c^2$, it follows that $\bar{f}\bar{e} \cdot u$ is a finite sum $\sum z_t$, where each $z_t$ is a product of $a$'s, $b$'s and $c$'s, but has at least one $a$ less than the original expression of $u$ had. By induction, for some $k$ we get that $(\bar{f}\bar{e})^k \cdot u$ is a finite sum of products of $b$'s and $c$'s.

Now observe that any product of $b$'s and $c$'s can be expressed as a linear combination of standard monomials $b^i c^j$, using the relation (\ref{equation:c_b}) and a very similar reasoning as before. The point is that $a$ does not appear in $[c,b]$ in (\ref{equation:c_b}), so we will not end up in an infinite loop.

Finally, note that the argument is essentially the same if we started from $u$ equal to a linear combination of products of the generators, instead of just one monomial.
\end{proof}

\begin{theorem} \label{theorem:annihilators}
Suppose $n \in \Z$ and $\lambda_2 \neq 0$. The annihilators in Corollary \ref{corollary:annihilators} are centrally generated. More precisely, they are equal to the ideal $U \cdot \Ker \chi$, where \[ \chi \colon Z(\T)=\C[C,\bar{C}] \to \C \] is a homomorphism of algebras defined on the generators by $C \mapsto n\lambda_2$ and $\bar{C} \mapsto \lambda_2^2$.
\end{theorem}
\begin{proof}
We prove this for the annihilator of the Verma module $\Delta := \Delta(n-2,\lambda_2)$. This is known from \cite[Proposition 6.1]{bavula2017prime}, but we present here a different and a more direct proof.

The inclusion $U \cdot \Ker \chi \subseteq \Ann(\Delta)$ is trivial. For the converse, recall that $\Ann(\Delta)$ is stable under the adjoint action, so it is generated by its $U_0$ part. So, it is enough to prove 
\[ U_0 \cap \Ann(\Delta) \subseteq U_0 \cdot \Ker \chi. \] 
To prove this, for any non-zero element $u \in \quotient{U_0}{U_0 \cdot \Ker \chi}$ we want to find an element from $\Delta$ which is not annihilated by $u$. Because of Lemma \ref{lemma:U0quotient_structure}, we can assume without loss of generality that
\[ u = \sum_{k,l,m \geq 0}  \alpha_{klm} h^k \left(f \bar{e}\right)^l \left(\bar{h}\right)^m, \]
with $\alpha_{klm} \in \C$ and only finitely many non-zero. Define a polynomial (with commutative variables) by the same scalars: $p(x,y,z)= \sum_{k,l,m \geq 0} \alpha_{klm} x^k y^l z^m \in \C[x,y,z]$.

Denote by $v$ the highest weight vector in $\Delta$, by $\Delta_{q}$ the weight space in $\Delta$ of weight $n-2-2q$, $q \geq 0$, and recall that it has basis $f^i \bar{f}^{q-i} v$ for $i=0,1,\ldots,q$. Similarly to (\ref{equation:e_Verma}), one can prove the following formulas for the action on $\Delta$:
\begin{align}
\nonumber h \cdot f^i \bar{f}^{q-i} v &= (n-2-2q) f^i \bar{f}^{q-i} v, \\
\nonumber f\bar{e} \cdot f^i \bar{f}^{q-i} v &= i \lambda_2 f^i \bar{f}^{q-i} v -i(i-1) f^{i-1} \bar{f}^{q-i+1} v, \\
\label{equation:barh_Verma} \bar{h} \cdot f^i \bar{f}^{q-i} v &= \lambda_2 f^i \bar{f}^{q-i} v -2i f^{i-1} \bar{f}^{q-i+1} v.
%
\end{align}

It follows that in this basis of $\Delta_q$, the operator representing $u$ is upper-triangular, with the diagonal entries $p(n-2-2q,i \lambda_2,\lambda_2)$, $i=0,\ldots,q$. We would like to find a basis element $f^i \bar{f}^{q-i} v$, for which $p(n-2-2q,i \lambda_2,\lambda_2) \neq 0$. However, a problem arises if $p(x,y,z)$ is divisible by $(z-\lambda_2)$.

We claim that we can decompose
\begin{equation} \label{equation:poly_division}
p(x,y,z) = \tilde{p}(x,y,z) \cdot (z-\lambda_2)^r
\end{equation}
for some $r \geq 0$, such that $\tilde{p}(n-2-2q,i \lambda_2,\lambda_2)$ is not identically zero for $(q,i) \in D$, where  $D \subseteq \C^2$ is any Zariski dense subset.

To prove this claim, write $p(x,y,z)= \sum_{j=0}^m p_j(x,y)(z-\lambda_2)^j$. Suppose that this is zero when evaluated on $D \times \{\lambda_2\}$ for a Zariski dense set $D \subseteq \C^2$. It follows that $p_0(x,y)=0$ (on $\C^2$), so $p(x,y,z)= p^{(1)}(x,y,z)(z-\lambda_2)$, for a polynomial $p^{(1)}(x,y,z)$ of a strictly smaller total degree. If necessary, we continue to apply the same argument inductively on $p^{(1)}(x,y,z)$, etc., until we reach (\ref{equation:poly_division}) with $\tilde{p}(x,y,\lambda_2)$ non-zero on some point in $(x,y) \in D$. The number $r$ is independent of $D$, since the set $\{(x,y) \colon \tilde{p}(x,y,\lambda_2) \neq 0\}$ is non-empty and Zariski open, hence intersects any Zariski dense set in $\C^2$.

The claim is now proved, because the map $(q,i) \mapsto (n-2-2q,i \lambda_2)$ is an algebraic isomorphism $\C^2 \to \C^2$. Here it is crucial that $\lambda_2 \neq 0$.

Write $\tilde{p}(x,y,z)= \sum_{k,l,m \geq 0} \tilde{\alpha}_{klm} x^k y^l z^m$, and define $\tilde{u}= \sum_{k,l,m \geq 0}  \tilde{\alpha}_{klm} h^k \left(f \bar{e}\right)^l \left(\bar{h}\right)^m$. Then it is also true that
\[ u = \tilde{u} \cdot (\bar{h}-\lambda_2)^r, \]
since the monomials in $u$ and $\tilde{u}$ have $\bar{h}$ on the most-right position, so no commuting of the variables is necessary.

There exists a pair $(q,i)$ from the cone $\{(q,i) \in \Z\times \Z \colon q\geq r, \ 0 \leq i \leq q-r \}$ (which is Zariski dense in $\C^2$), such that $\tilde{p}(n-2-2q,i \lambda_2,\lambda_2) \neq 0$. Put $w:=f^{i+r} \bar{f}^{q-i-r}v \in \Delta_q$. It follows from (\ref{equation:barh_Verma}) that $(\bar{h}-\lambda_2)^r \cdot w = c \cdot f^{i} \bar{f}^{q-i}$, for some constant $c \neq 0$. From this we have that 
\begin{align*}
u \cdot w &= c \cdot \tilde{u} \cdot f^{i} \bar{f}^{q-i}v \\
&= c \cdot \tilde{p}(n-2-2q,i \lambda_2,\lambda_2) \cdot f^{i} \bar{f}^{q-i}v + \sum_{j=0}^{i-1} c_j f^{j} \bar{f}^{q-j}v \neq 0.  \qedhere
\end{align*}
\end{proof}

\subsection{The action of finite-dimensional $\mathfrak{sl}_2$-modules}

Denote by $\cF$ the monoidal category of finite-dimensional $\mathfrak{sl}_2$-modules.
For a fixed non-zero $\chi \in \C$, denote by $\mathcal{H}_{\chi}$ the category of 
semi-simple $\mathfrak{g}$-Harish-Chandra $\T$-modules on which the action of the
purely Takiff part of the center is given by $\chi$.

\begin{proposition}\label{propact1}
For each non-zero $\chi$, the category $\mathcal{H}_{\chi}$ is a simple module
category over $\cF$.
\end{proposition}

\begin{proof}
The fact that $\mathcal{H}_{\chi}$ is a module
category over $\cF$ follows directly from Proposition~\ref{proposition:V_otimes_L}.
Since $\mathcal{H}_{\chi}$ is semi-simple by definition, to show that it is a simple
module category over $\cF$ it is enough to show that, staring from any simple
object of $\mathcal{H}_{\chi}$ and tensoring it with finite-dimensional 
$\mathfrak{sl}_2$-modules, we can obtain any other simple object of
$\mathcal{H}_{\chi}$ as a direct summand, up to isomorphism. This claim follows by combining
Proposition~\ref{proposition:V_otimes_L} with Theorem~\ref{theorem:classification}.
\end{proof}

We note that, by Proposition~\ref{proposition:V_otimes_L}, the combinatorics
of the $\cF$-module category $\mathcal{H}_{\chi}$ does not depend on $\chi$.

\section{$\mathfrak{sl}_2$-Harish-Chandra modules for the Schr\"{o}dinger Lie algebra}
\label{section:schrodinger}

\subsection{Setup}

The \emph{Schr\"{o}dinger Lie algebra} $\s$ can be defined by basis $\{e,h,f,p,q,z\}$ and the following relations: in addition to the usual $\g :=\mathfrak{sl}_2$ relations on $e,h,f$, we also have
\begin{align*}
&[e,p] = 0, &  &[h,p] = p, &     &[f,p] = q, &&  \\
&[e,q] = p, &  &[h,q] = -q, &     &[f,q] = 0, && [p,q]=z.
\end{align*}
and $z$ is declared to commute with all $\s$. It is clear that $\s = \g \ltimes \HH$, where $\HH$ is the ideal spanned by $p,q,z$, and is isomorphic to the 3-dimensional Heisenberg Lie algebra. As a $\g$-module, $\HH$ is isomorphic to $L(1) \oplus L(0)$.

The nilradical of $\s$ is $\Nrad(\s)=[\s,\HH]=\HH$. Recall that this means that $\HH$ must necessarily annihilate any simple finite-dimensional $\s$-module.

There is also the \emph{centerless Schr\"{o}dinger Lie algebra} $\bar{\s} := \quotient{\s}{\C z}$, which is isomorphic to the semi-direct product $\g \ltimes L(1)$.

The disclaimer from the previous section related to \cite{bavula2017prime}
applies to the present section with respect to \cite{bavula2018classification}.

The algebra $U(\s)$ is free as a module over its center $Z(\s)$, and $Z(\s)$ is generated by two algebraically independent generators (see e.g. \cite{dubsky2014category}):
\begin{equation}
\label{equation:center_sch}
C := (h^2+h+4fe)z - 2(fp^2 -eq^2-hpq), \text{ and } z.
\end{equation}
It is also clear that $Z(\s) \cap U(\HH) = \C[z]$, which we will refer to as the \emph{purely Schr\"{o}dinger part of the center}. For a module with central character, the scalar by which $z$ acts is usually called the \emph{central charge} of the module.

The theory we develop here for the Schr\"{o}dinger Lie algebra is very similar to the Takiff $\mathfrak{sl}_2$ case. So we will omit most of the details, as they are usually analogous, but easier. One reason for this is that the purely Schr\"{o}dinger part of the center is generated by a degree $1$ element, and for the Takiff $\mathfrak{sl}_2$ we had a degree $2$ element. However, a small complication now is that the radical of $\s$ is not abelian anymore.

\subsection{Universal modules}

As before, the universal modules are induced from $\g$, i.e., for $n \in \Z_{\geq 0}$ set $Q(n) := \Ind_{\g}^{\s} L(n) = U(\s) \tens{U(\g)} L(n) \cong U(\HH) \tens{\C} L(n)$. Recall that $Q(n) \cong Q(0) \tens{} L(n)$, where we consider $L(n)$ as an $\s$-module with the trivial $\HH$-action. 

\begin{proposition} \label{proposition:EndQ0_sch}
We have the following isomorphisms of algebras:
\begin{equation*}
\End(Q(0)) \cong U(\HH)^\g = \C[z],
\end{equation*}
where $U(\HH)^\g$ denotes the invariants of the adjoint action of $\g$ on $U(\HH)$.
\end{proposition}
\begin{proof}
The isomorphism $\End(Q(0))^\text{op} \cong U(\HH)^\g$ follows from the same argument as in the proof of Proposition \ref{proposition:end_Q0}. The inclusion $U(\HH)^\g \supseteq \C[z]$ is obvious. The converse follows easily from the following commutation relations, which can be proved e.g. by induction:
\begin{align}
\label{equation:g_pq} [h,p^mq^n] &= (m-n) p^mq^n, \\
\nonumber [e,p^mq^n] &= n p^{m+1} q^{n-1} - \frac{n(n-1)}{2} p^{m} q^{n-2}z, \\
\nonumber [f,p^mq^n] &= m p^{m-1} q^{n+1} - \frac{m(m-1)}{2} p^{m-2} q^{n}z. \qedhere
\end{align}
\end{proof}

Fix $\chi \in \C$, and denote by $\mathbf{m}_\chi$ the maximal ideal $(z-\chi) \subseteq \C[z]$. As before, we define the \emph{universal module} as $Q(n,\chi) := \quotient{Q(n)}{\mathbf{m}_\chi Q(n)}$. It clearly has central charge $\chi$. As before, we have $Q(n,\chi) \cong Q(0,\chi) \otimes L(n)$.

\begin{lemma} \label{lemma:structure_Q0_sch}
\begin{enumerate}[(a)]
\item As $\s$-modules, $Q(0) \cong U(\HH)$ and $Q(0,\chi) \cong \quotient{U(\HH)}{(z-\chi)}$, where $\g$ acts by the adjoint action, and $\HH$ by the left multiplication. The set $\big\{p^i q^j \colon i,j \geq 0\big\}$ is a basis for $Q(0,\chi)$.

\item As a $\g$-module, $Q(0,\chi) \cong \bigoplus_{k\geq 0} L(k)$. Moreover, $p^k$ is the highest weight vector in $L(k)$.

\item $C$ acts as zero on $Q(0)$ and every $Q(0,\chi)$.
\end{enumerate}
\end{lemma}
\begin{proof}
The first claim is clear. We use it to prove the others.

For the second claim, note that $p^k$ generates a $\g$-submodule isomorphic to $L(k)$. Since the action of $\g$ preserves $Q^n :=\operatorname{span}\{p^i q^j \colon i+j \leq n\}$, by counting dimensions we see that $Q^n \cong \oplus_{k=0}^n L(k)$. The claim now follows by taking colimits.

The last claim can be checked directly (enough on the generator of $Q(0)$).
\end{proof}

From the previous lemma, the Clebsch–Gordan coefficients for $\mathfrak{sl}_2$, and the adjunction, the following is not hard to deduce:

\begin{proposition} \label{proposition:Q_sch}
\begin{enumerate}[(a)]
\item $Q(n,\chi)$ is a $\g$-Harish-Chandra module, and for $k \geq 0$:
\begin{equation} \label{equation:multiplicity_shr}
[Q(n,\chi) \colon L(k)] = \min\{k+1,n+1\}.
\end{equation}

\item 
Let $V$ be any simple $\s$-module that has some $L(n)$ as a simple $\g$-submodule. Then $V$ is a quotient of $Q(n,\chi)$ for a unique $\chi$. In particular, $V$ is a $\g$-Harish-Chandra module, and (\ref{equation:multiplicity_shr}) gives an upper bound for the multiplicities of its $\g$-types.

\item \label{item:uniqueQ0}
For a fixed $\chi$, there exists a unique simple $\s$-module which contains $L(0)$ as a $\g$-submodule and has central charge $\chi$. Moreover, it is a $\g$-Harish-Chandra module.
\end{enumerate}
\end{proposition}

\begin{theorem} \label{theorem:Q0_sch}
The module $Q(0,\chi)$ is simple if and only if $\chi \neq 0$.

The module $Q(0,0)$ has infinite length, and an $\s$-filtration whose composition factors are $L(0), L(1), L(2) \ldots$ with the trivial action of $\HH$. 
\end{theorem}
\begin{proof}
As before, the $\s$-action on $Q(0)$ and $Q(0,\chi)$ will be denoted by $\circ$.

Note that $[q,p^n] = -n p^{n-1}z$ and $[p,q^n] = n q^{n-1}z$. Using this and the equations (\ref{equation:g_pq}), one can check that
\[ (pf - nq) \circ p^n = \frac{n(n+1)}{2}\chi \cdot p^{n-1}. \]
So if $\chi \neq 0$, the module $Q(0,\chi)$ is simple.

Alternatively, one can use a nilradical argument analogous to the one in Remark \ref{remark:simpllicity_Q0}.

If $\chi=0$, then $p$ and $q$ commute in $Q(0,\chi)$, and $\g$ preserves the total degree of monomials $p^iq^j$.  The rest of the proof is obvious.
\end{proof}

\subsection{Verma modules}

Verma modules for the Schr\"{o}dinger Lie algebra are studied in detail in \cite{dubsky2014category}.

In the triangular decomposition (\ref{equation:general_triangular}) we have
\[ \tilde{\n}_- = \operatorname{span}\{f,q\}, \ \tilde{\h} = \operatorname{span}\{h,z\}, \ \text{and} \ \tilde{\n}_+ = \operatorname{span}\{e,p\}. \]
For an element $\lambda \in \tilde{\h}^\ast$, denote $\lambda_1 := \lambda(h)$ and $\lambda_2 := \lambda(z)$.
\begin{proposition}[{\cite[Proposition 5]{dubsky2014category}}]
If $\lambda_2 \neq 0$, then the Verma module $\Delta(\lambda)$ is simple for any $\lambda_1 \in \Z$.
\end{proposition}

It is easy to see that the central element $C$ acts on the Verma module $\Delta(\lambda)$ as the scalar $(\lambda_1+1)(\lambda_1+2)\lambda_2$, and the central charge is $\chi := \lambda_2$, see (\ref{equation:center_sch}). We will be concerned mostly with non-zero central charge cases. 
\begin{lemma} \label{lemma:Vermas_sch_cen_ch}
Non-isomorphic Verma modules $\Delta(\lambda)$ and $\Delta(\lambda')$ with the same non-zero central charge have the same central character if and only if $\lambda'_1 = -\lambda_1-3$.
\end{lemma}
\begin{proof}
This reduces to solving the equation $(\lambda_1+1)(\lambda_1+2) = (\lambda'_1+1)(\lambda'_1+2)$.
\end{proof}

\begin{lemma} \label{lemma:structure_Verma_sch}
As a $\g$-module, $\Delta(\lambda)$ has a filtration with subquotients isomorphic to the $\g$-Verma modules $\Delta^\g(\lambda_1 - k)$, $k = 0, 1, 2, \ldots$

If $\lambda_2 = 0$ or $\lambda_1 \not\in \Z_{\geq 0}$, then as a $\g$-module we have $\Delta(\lambda) \cong \bigoplus_{k \geq 0} \Delta^\g(\lambda_1 - k)$.
Otherwise ($\lambda_2 \neq 0$ and $\lambda_1 \in \Z_{\geq 0}$) we have as $\g$-modules
\[ \Delta(\lambda) \cong \Delta^\g(-1)\oplus \bigoplus_{k =2}^{\lambda_1+2} P^\g(-k) \oplus \bigoplus_{k \geq 3} \Delta^\g(-\lambda_1 - k). \]
\end{lemma}
\begin{proof}
Denote by $v_\lambda$ a basis element of $\C_\lambda$. Then $\Delta(\lambda)$ has a basis of weight vectors $\{ f^i q^j v_\lambda \colon i,j \geq 0\}$. A direct computation shows that
\[ e \cdot f^i q^j v_\lambda = i(\lambda_1 - i -j +1) f^{i-1} q^j v_\lambda + \lambda_2 \frac{j(j-1)}{2} f^i q^{j-2} v_\lambda. \]
This implies that the required filtration is given by the degree of $q$. The subquotients are given by the span of $\{ f^i q^k v_\lambda \colon i \geq 0\}$, which is clearly isomorphic to $\Delta^\g(\lambda_1 - k)$. The rest can be proved in the same way as for Lemma \ref{lemma:structure_Verma}.
\end{proof}

\begin{lemma} \label{lemma:verma_otimes_fin_sch}
For $\lambda_1 \in \Z$, $\lambda_2 \neq 0$ and $\mu \in \Z_{\geq 0}$ there is an isomorphism of $\s$-modules 
\[ \Delta(\lambda) \tens{} L(\mu) \cong \Delta(\lambda_1+\mu,\lambda_2) \oplus \Delta(\lambda_1+\mu-2,\lambda_2) \oplus \ldots \oplus \Delta(\lambda_1-\mu,\lambda_2). \]
\end{lemma}
\begin{proof}
The left-hand side has a filtration with subquotients equal to the summands on the right-hand side. But these subquotients have different central characters by Lemma \ref{lemma:Vermas_sch_cen_ch}, since the first components of their highest weights have the same parity, so they must split.
\end{proof}

\subsection{Enright-Arkhipov completion}

Fix an $\ad$-nilpotent element $x \in \s$ (for example $f$ or $p$, which we will use), and denote by $S_x := \quotient{U(\s)_{(x)}}{U(\s)}$ the localization of the algebra $U(\s)$ by $x$, modulo the canonical copy of $U(\s)$ inside it. This is a $U(\s)$-bimodule. For an $\s$-module $M$ write
\[ \EnAr(M) := \prescript{e}{}{\left( S_f \tens{U(\s)} M \right)}. \]
As before, one can check that this is a well defined functor on the category of $\s$-modules. Moreover, Proposition \ref{proposition:EA_tensor_L} is valid here, with the same proof.

\begin{theorem} \label{theorem:EA_Verma_sch}
Take $\lambda$ with $\lambda_1 \in \Z$ and $\lambda_2 \neq 0$. Then $\EnAr(\Delta(\lambda))$ is a simple $\g$-Harish-Chandra module, and decomposes as a $\g$-module as
\begin{equation} \label{equation:g_types_EA_sch}
\EnAr(\Delta(\lambda)) \cong \bigoplus_{k \geq 0} L\left(\left|\lambda_1 + \frac{3}{2}\right| - \frac{1}{2} +k \right).
\end{equation}
\end{theorem}
\begin{proof}
Lemma \ref{lemma:structure_Verma_sch}, the fact that the functor $\EnAr$ commutes with the forgetful functor from $\s$-modules to $\g$-modules, and (the Schr\"{o}dinger analogue of) Example \ref{example:EA_Delta_g} imply the decomposition (\ref{equation:g_types_EA_sch}).

Note that the lowest weight vector of a $L(\mu)$ in (\ref{equation:g_types_EA_sch}) is $f^{-1} q^t v_\lambda$, where $t:=2+\lambda_1+\mu$. To prove simplicity, it is enough for $\mu \geq \left| \lambda_1 + \frac{3}{2}\right| + \frac{1}{2}$ to find an element in $U(\s)$ that maps $f^{-1} q^t v_\lambda$ to  $f^{-1} q^{t-1} v_\lambda$. One can check by direct calculation that
\[ \left( p - \frac{1}{\mu} qe \right) \cdot f^{-1} q^t v_\lambda =  \frac{\lambda_2 t}{2 \mu}(\mu-\lambda_1-1) f^{-1} q^{t-1} v_\lambda. \]
The scalar on the right-hand side is non-zero because of the assumption on $\mu$.

Alternatively, one can use a nilradical argument analogous to the one in Remark \ref{remark:simpllicity_EA}.
\end{proof}

Theorem \ref{theorem:EA_Verma_sch}, Proposition \ref{proposition:Q_sch}(\ref{item:uniqueQ0}) and Theorem \ref{theorem:Q0_sch} together give:
\begin{corollary} \label{corollary:Q0_V0_sch}
For $\lambda_2 \neq 0$ we have
\[\EnAr(\Delta(-1,\lambda_2)) \cong \EnAr(\Delta(-2,\lambda_2)) \cong Q(0,\lambda_2). \]
\end{corollary}

\subsection{Classification}
The Enright-Arkhipov completion of Verma modules again gives us a family of $\g$-Harish-Chandra modules. This construction gives all infinite-dimensional $\g$-Harish-Chandra modules, as we will see in this subsection.

\begin{definition}
For $n \in \Z$ and $\lambda_2 \neq 0$ denote
\begin{equation*}
V(n,\lambda_2) := \begin{cases} \EnAr(\Delta(n-1, \lambda_2)) &\colon n \geq 0, \\
\EnAr(\Delta(n-2, \lambda_2)) &\colon n \leq 0. \end{cases}
\end{equation*} 
\end{definition}
From Corollary \ref{corollary:Q0_V0_sch} we have that $V(0,\lambda_2)$ is well-defined, and moreover isomorphic to $Q(0,\lambda_2)$. Note that the central element $C$ acts on $V(n,\lambda_2)$ as $n(n+1)\lambda_2$ if $n \geq0$, and as $n(n-1)\lambda_2$ if $n \leq0 $.

Theorem \ref{theorem:EA_Verma_sch} and Lemma \ref{lemma:Vermas_sch_cen_ch} easily give:

\begin{corollary} \label{corollary:g_types_sch}
The module $V(n,\lambda_2)$ is a simple $\g$-Harish-Chandra module, and has $\g$-types $L(|n|), L(|n|+1), L(|n|+2) \ldots$ with multiplicity one.

If $V(n,\lambda_2) \cong V(n',\lambda_2')$, then $\lambda_2'=\lambda_2$ and $n' \in \{n,-n\}$.
\end{corollary}

On Figure \ref{figure:V_1_sch} to \ref{figure:V2_sch} we present several $V(n,\lambda_2)$'s, and how they are constructed. It is interesting to compare this to the Takiff $\mathfrak{sl}_2$ case (cf. Figure \ref{figure:V_1} to \ref{figure:V2}).
\begin{figure}[h]
\begin{tikzpicture}[scale=.5] 
\def\epsilon{.2}

\newcommand{\Verma}[6]{
	\draw [fill=gray!30] (#1-2*\epsilon,#2+2*\epsilon) --(#3+2*2*\epsilon,#4+2*\epsilon) -- (#5-2*\epsilon,#6-2*2*\epsilon);
	}

\newcommand{\arrow}[4]{
	\draw [->] (#1-\epsilon*#1+\epsilon*#3,#2-\epsilon*#2+\epsilon*#4) -- (\epsilon*#1+#3-\epsilon*#3,\epsilon*#2+#4-\epsilon*#4);
	}

\Verma{-5}{0}{-3}{0}{-5}{-2};

\path [fill=gray!15] (5+2*\epsilon,0+2*\epsilon) --(3-2*2*\epsilon,0+2*\epsilon) -- (5+2*\epsilon,-2-2*2*\epsilon);

\foreach \x in {-5,-3,-1,1,3,5}{
	\foreach \y in {0,-2}{ 
    	\node[draw,circle,inner sep=1pt,fill] at (\x,\y) {};
  		}
	}
\foreach \x in {-4,-2,0,2,4}{
	\node[draw,circle,inner sep=1pt,fill] at (\x,-1) {};
  		
	}
	
\foreach \y in {0,-2}{ 
	\node[] at (-6,\y) {$\ldots$};
	\node[] at (6,\y) {$\ldots$};
	}
\node[] at (-6,-1) {$\ldots$};
\node[] at (6,-1) {$\ldots$};

\node[] at (0,-2.75) {$\ldots$};

\foreach \x in {-5,-3,-1,1,3,5}{
	\node[above,yshift=5] at (\x,0) {{\scriptsize $\x$}};
	}

\arrow{-6}{0}{-5}{0};\arrow{-5}{0}{-3}{0};\arrow{-1}{0}{1}{0};\arrow{3}{0}{5}{0};\arrow{5}{0}{6}{0};

\arrow{-6}{-1}{-4}{-1};\arrow{-2}{-1}{0}{-1};\arrow{0}{-1}{2}{-1};\arrow{4}{-1}{6}{-1};

\arrow{-6}{-2}{-5}{-2};\arrow{-3}{-2}{-1}{-2};\arrow{-1}{-2}{1}{-2};\arrow{1}{-2}{3}{-2};\arrow{5}{-2}{6}{-2};
\end{tikzpicture}
\caption{$V(-1,\lambda_2) = \EnAr(\Delta(-3,\lambda_2))$}
\label{figure:V_1_sch}
\bigskip\bigskip
\begin{tikzpicture}[scale=.5] 
\def\epsilon{.2}

\newcommand{\Verma}[6]{
	\draw [fill=gray!30] (#1-2*\epsilon,#2+2*\epsilon) --(#3+2*2*\epsilon,#4+2*\epsilon) -- (#5-2*\epsilon,#6-2*2*\epsilon);
	}

\newcommand{\arrow}[4]{
	\draw [->] (#1-\epsilon*#1+\epsilon*#3,#2-\epsilon*#2+\epsilon*#4) -- (\epsilon*#1+#3-\epsilon*#3,\epsilon*#2+#4-\epsilon*#4);
	}

\Verma{-4}{0}{-2}{0}{-4}{-2};

\path [fill=gray!15] (4+2*\epsilon,0+2*\epsilon) --(2-2*2*\epsilon,0+2*\epsilon) -- (4+2*\epsilon,-2-2*2*\epsilon);

\foreach \x in {-4,-2,0,2,4}{
	\foreach \y in {0,-2}{ 
    	\node[draw,circle,inner sep=1pt,fill] at (\x,\y) {};
  		}
	}
\foreach \x in {-3,-1,1,3}{
	\node[draw,circle,inner sep=1pt,fill] at (\x,-1) {};
  		
	}
	
\foreach \y in {0,-2}{ 
	\node[] at (-5,\y) {$\ldots$};
	\node[] at (5,\y) {$\ldots$};
	}
\node[] at (-5,-1) {$\ldots$};
\node[] at (5,-1) {$\ldots$};

\node[] at (0,-2.75) {$\ldots$};

\foreach \x in {-4,-2,0,2,4}{
	\node[above,yshift=5] at (\x,0) {{\scriptsize $\x$}};
	}

\arrow{-5}{0}{-4}{0};\arrow{-4}{0}{-2}{0};\arrow{2}{0}{4}{0};\arrow{4}{0}{5}{0};

\arrow{-5}{-1}{-3}{-1};\arrow{-1}{-1}{1}{-1};\arrow{3}{-1}{5}{-1};

\arrow{-5}{-2}{-4}{-2};\arrow{-2}{-2}{0}{-2};\arrow{0}{-2}{2}{-2};\arrow{4}{-2}{5}{-2};
\end{tikzpicture}
\caption{$V(0,\lambda_2) = \EnAr(\Delta(-2,\lambda_2))$}
\label{figure:V0_sch}
\bigskip\bigskip
\begin{tikzpicture}[scale=.5] 
\def\epsilon{.2}

\newcommand{\Verma}[6]{
	\draw [fill=gray!30] (#1-2*\epsilon,#2+2*\epsilon) --(#3+2*2*\epsilon,#4+2*\epsilon) -- (#5-2*\epsilon,#6-2*2*\epsilon);
	}

\newcommand{\arrow}[4]{
	\draw [->] (#1-\epsilon*#1+\epsilon*#3,#2-\epsilon*#2+\epsilon*#4) -- (\epsilon*#1+#3-\epsilon*#3,\epsilon*#2+#4-\epsilon*#4);
	}

\Verma{-4}{0}{-1}{0}{-4}{-3};

\path [fill=gray!15] (4+2*\epsilon,0+2*\epsilon) --(1-2*2*\epsilon,0+2*\epsilon) -- (4+2*\epsilon,-3-2*2*\epsilon);

\foreach \x in {-3,-1,1,3}{
	\foreach \y in {0,-2}{ 
    	\node[draw,circle,inner sep=1pt,fill] at (\x,\y) {};
  		}
	}
\foreach \x in {-4,-2,0,2,4}{
	\foreach \y in {-1,-3}{
		\node[draw,circle,inner sep=1pt,fill] at (\x,\y) {};
  		}
	}
	
\foreach \y in {0,-1,-2,-3}{ 
	\node[] at (-5,\y) {$\ldots$};
	\node[] at (5,\y) {$\ldots$};
	}

\node[] at (0,-3.75) {$\ldots$};

\foreach \x in {-3,-1,1,3}{
	\node[above,yshift=5] at (\x,0) {{\scriptsize $\x$}};
	}

\arrow{-5}{0}{-3}{0};\arrow{-3}{0}{-1}{0};\arrow{1}{0}{3}{0};\arrow{3}{0}{5}{0};

\arrow{-5}{-1}{-4}{-1};\arrow{-4}{-1}{-2}{-1};\arrow{2}{-1}{4}{-1}\arrow{4}{-1}{5}{-1};

\arrow{-5}{-2}{-3}{-2};\arrow{-1}{-2}{1}{-2};\arrow{3}{-2}{5}{-2};

\arrow{-5}{-3}{-4}{-3};\arrow{-2}{-3}{0}{-3};\arrow{0}{-3}{2}{-3};\arrow{4}{-3}{5}{-3};
\end{tikzpicture}
\caption{$V(0,\lambda_2) = \EnAr(\Delta(-1,\lambda_2))$}
\label{figure:V1_sch}
\bigskip\bigskip
\begin{tikzpicture}[scale=.5] 
\def\epsilon{.2}

\newcommand{\Verma}[6]{
	\draw [fill=gray!30] (#1-2*\epsilon,#2+2*\epsilon) --(#3+2*2*\epsilon,#4+2*\epsilon) -- (#5-2*\epsilon,#6-2*2*\epsilon);
	}

\newcommand{\arrow}[4]{
	\draw [->] (#1-\epsilon*#1+\epsilon*#3,#2-\epsilon*#2+\epsilon*#4) -- (\epsilon*#1+#3-\epsilon*#3,\epsilon*#2+#4-\epsilon*#4);
	}

\newcommand{\roundarrow}[5]{
	\draw [->] (#1-.1-\epsilon/2*#1+\epsilon/2*#3,#2-\epsilon/2*#2+\epsilon/2*#4) to [out=45+15,in=-135-15] (\epsilon/2*#1+#3-\epsilon/2*#3,\epsilon/2*#2+#4-\epsilon/2*#4+.1);
	}

\Verma{-4}{0}{0}{0}{-4}{-4};

\path [fill=gray!15] (4+2*\epsilon,0+2*\epsilon) -- (2-2*\epsilon,0+2*\epsilon) -- (0-3*2*\epsilon,-2-2*\epsilon) -- (2-2*\epsilon + \epsilon,-2-2*\epsilon) -- (4+2*\epsilon,-4-3*2*\epsilon+\epsilon);

\foreach \x in {-4,-2,0,2,4}{
	\foreach \y in {0,-2,-4}{ 
    	\node[draw,circle,inner sep=1pt,fill] at (\x,\y) {};
  		}
	}
\foreach \x in {-3,-1,1,3}{
	\foreach \y in {-1,-3}{
		\node[draw,circle,inner sep=1pt,fill] at (\x,\y) {};
  		}
	}
	
\foreach \y in {0,-1,-2,-3,-4}{ 
	\node[] at (-5,\y) {$\ldots$};
	\node[] at (5,\y) {$\ldots$};
	}

\node[] at (0,-4.75) {$\ldots$};

\foreach \x in {-4,-2,0,2,4}{
	\node[above,yshift=5] at (\x,0) {{\scriptsize $\x$}};
	}

\arrow{-5}{0}{-4}{0};\arrow{-4}{0}{-2}{0};\arrow{2}{0}{4}{0};\arrow{4}{0}{5}{0};

\arrow{-5}{-1}{-3}{-1};\arrow{-3}{-1}{-1}{-1};\arrow{1}{-1}{3}{-1};\arrow{3}{-1}{5}{-1};

\arrow{-5}{-2}{-4}{-2};\arrow{-4}{-2}{-2}{-2};\arrow{2}{-2}{4}{-2};\arrow{4}{-2}{5}{-2};

\arrow{-5}{-3}{-3}{-3};\arrow{-1}{-3}{1}{-3};\arrow{3}{-3}{5}{-3};

\arrow{-5}{-4}{-4}{-4};\arrow{-2}{-4}{0}{-4};\arrow{0}{-4}{2}{-4};\arrow{4}{-4}{5}{-4};

\roundarrow{-2}{-2}{0}{0};
\roundarrow{0}{-2}{2}{0};
\roundarrow{2}{-2}{4}{0};
\end{tikzpicture}
\caption{$V(1,\lambda_2) = \EnAr(\Delta(0,\lambda_2))$}
\label{figure:V2_sch}
\end{figure}

\begin{proposition} \label{proposition:Q_n_sch}
For $n \in \Z_{\geq 0}$ and $\lambda_2 \neq 0$ we have $V(-n,\lambda_2) \cong V(n,\lambda_2)$. Moreover,
\[ Q(n,\lambda_2) \cong V(n,\lambda_2) \oplus V(n-1,\lambda_2) \oplus \ldots \oplus V(0,\lambda_2). \]
\end{proposition}
\begin{proof}
We use induction over $n$. The basis is given in Corollary \ref{corollary:Q0_V0_sch}. Suppose the proposition is true for all $k=0,\ldots,n-1$ where $n\geq 1$ is fixed. Observe that, using (the Schr\"{o}dinger version of) Proposition \ref{proposition:EA_tensor_L} and Lemma \ref{lemma:verma_otimes_fin_sch}, we have
\begin{align}
\label{align:Q_n_sch} Q(n,\lambda_2) \cong & \, Q(0,\lambda_2) \otimes L(n) \cong \EnAr(\Delta(-1,\lambda_2)) \otimes L(n) \\
\nonumber \cong & \EnAr(\Delta(n-1,\lambda_2)) \oplus \EnAr(\Delta(n-3,\lambda_2)) \oplus \ldots \oplus \EnAr(\Delta(-n-1,\lambda_2)) \\
\nonumber \cong & \, V(n,\lambda_2) \oplus V(n-2,\lambda_2) \oplus \ldots \oplus V(\epsilon,\lambda_2) \oplus{} \\
\nonumber & \oplus V(\epsilon-1,\lambda_2) \oplus V(\epsilon-3,\lambda_2) \oplus \ldots \oplus V(-n+1,\lambda_2), \\
\intertext{where $\epsilon \in \{0,1\}$ is of the same parity as $n$. By inductive assumption it follows that}
\nonumber Q(n,\lambda_2) \cong & \, V(n,\lambda_2) \oplus V(n-1,\lambda_2) \oplus \ldots \oplus V(1,\lambda_2) \oplus V(0,\lambda_2) \\
\nonumber \cong & \, V(n,\lambda_2) \oplus Q(n-1,\lambda_2)
\end{align}

In the same way, but using $Q(0,\lambda_2) \cong \EnAr(\Delta(-2,\lambda_2))$ in the first line (\ref{align:Q_n_sch}) we can get that $Q(n,\lambda_2) \cong V(-n,\lambda_2) \oplus Q(n-1,\lambda_2)$. It follows that $V(n,\lambda_2) \cong V(-n,\lambda_2)$.
\end{proof}

Similarly, one can prove the following analogue of Proposition \ref{proposition:V_otimes_L}:
\begin{proposition}\label{prop12}
Let $n,k \in \Z_{\geq 0}$ and $\lambda_2 \neq 0$. If $k \leq n$, then 
\[ V(n,\lambda_2) \tens{} L(k) = V(n-k,\lambda_2) \oplus V(n-k+2,\lambda_2) \oplus \ldots \oplus V(n+k,\lambda_2). \]
If $k>n$, then
\begin{multline*}
V(n,\lambda_2) \tens{} L(k) = V(0,\lambda_2) \oplus V(1,\lambda_2) \oplus \ldots \oplus V(k-n-1,\lambda_2) \oplus {} \\ {} \oplus V(k-n,\lambda_2) \oplus V(k-n+2,\lambda_2) \oplus \ldots \oplus V(k+n,\lambda_2).
\end{multline*}
\end{proposition}

Since any simple $\g$-Harish-Chandra module is a quotient of some $Q(n,\lambda_2)$, we have proved:

\begin{theorem} \label{theorem:classification_sch}
Let $V$ be a simple $\g$-Harish-Chandra module for $\s$. Denote by $\lambda_2$ its central charge, and suppose $L(n)$, $n \in \Z_{\geq 0}$, is the minimal $\g$-type of $V$.
\begin{itemize}
\item If $\lambda_2 \neq 0$, then $V \cong V(n,\lambda_2) \cong V(-n,\lambda_2)$.
\item If $\lambda_2 = 0$, then $V \cong L(n)$ with the trivial $\HH$-action.
\end{itemize}
In other words, $V(n,\lambda_2)$, $n \in \Z_{\geq 0}$, $\lambda_2 \in \C\setminus\{0\}$, together with the finite-dimensional simple $\g$-modules constitute a complete list of pairwise non-isomorphic simple $\g$-Harish-Chandra modules for $\s$.
\end{theorem}

\begin{remark} \label{remark:centerless}
For the centerless Schr\"{o}dinger Lie algebra $\bar{\s}$, infinite-dimensional simple $\g$-Harish-Chandra modules do not exist. All simple $\g$-Harish-Chandra modules are given by $L(n)$, $n \in \Z_{\geq 0}$, with the trivial action of $\bar{\HH} := \quotient{\HH}{\C z}$.

This follows from observing that the endomorphism ring of $\Ind_{\g}^{\bar{\s}}(L(0))$ is only $\C$ (similarly as in Proposition \ref{proposition:EndQ0_sch}), and so all the universal modules have $\bar{\s}$-filtrations by simple finite-dimensional modules (similarly as in Theorem \ref{theorem:Q0_sch} for $\chi=0$).
\end{remark}

\subsection{Extensions}

Self-extensions of infinite-dimensional simple $\g$-Harish-Chandra modules (i.e. the ones having non-zero central charge) can be calculated in the same way as for the Takiff $\mathfrak{sl}_2$ case (Subsection \ref{subsection:extensions}):

\begin{theorem} \label{theorem:extensions_sch}
Let $V$ be a simple infinite-dimensional $\g$-Harish-Chandra module for $\s$. Then $\Ext^1(V,V) \cong \C$.
\end{theorem}

\subsection{Annihilators} We show that simple infinite-dimensional $\g$-Harish-Chandra modules again have the same annihilators as the corresponding Verma modules. The fact that annihilators of Verma modules for $\s$ are centrally generated is already known, see \cite[Theorem 21]{dubsky2014category}.

\begin{proposition} \label{proposition:V_localized_sch}
Suppose $n \in \Z_{\geq 0}$ and $\lambda_2 \neq 0$.
\begin{enumerate}[(a)] 
\item The element $p$ acts injectively on $V(n,\lambda_2)$.
\item The module $S_{p} \tens{U(\s)} V(n,\lambda_2)$ is isomorphic to the direct sum of Verma modules $\Delta(n-1,\lambda_2) \oplus \Delta(-n-2,\lambda_2)$.
\end{enumerate}
\end{proposition}
\begin{proof}
The proof is analogous to the proof of Proposition \ref{proposition:V_localized}, but easier. So we will omit it.
\end{proof}

From (the Schr\"{o}dinger version of) Lemma \ref{lemma:ann_loc}, Proposition \ref{proposition:V_localized_sch}, the definition of $V(n,\lambda_2)$, and \cite[Theorem 21]{dubsky2014category} we have:

\begin{corollary} \label{corollary:annihilators_sch}
Suppose $n \in \Z_{\geq 0}$ and $\lambda_2 \neq 0$. Then 
\[ \Ann(V(n,\lambda_2)) = \Ann(\Delta(n-1,\lambda_2)) = \Ann(\Delta(-n-2,\lambda_2)), \]
and these annihilators are centrally generated.
\end{corollary}

\subsection{The action of finite-dimensional $\mathfrak{sl}_2$-modules}

Denote by $\cF$ the monoidal category of finite-dimensional $\mathfrak{sl}_2$-modules.
For a fixed non-zero $\chi \in \C$, denote by $\mathcal{K}_{\chi}$ the category of 
semi-simple $\mathfrak{g}$-Harish-Chandra $\mathfrak{s}$-modules of central charge $\chi$.

\begin{proposition}\label{propact2}
For each non-zero $\chi$, the category $\mathcal{K}_{\chi}$ is a simple module
category over $\cF$.
\end{proposition}

\begin{proof}
The fact that $\mathcal{K}_{\chi}$ is a module
category over $\cF$ follows directly from Proposition~\ref{prop12}.
Since $\mathcal{K}_{\chi}$ is semi-simple by definition, to show that it is a simple
module category over $\cF$ it is enough to show that, staring from any simple
object of $\mathcal{K}_{\chi}$ and tensoring it with finite-dimensional 
$\mathfrak{sl}_2$-modules, we can obtain any other simple object of
$\mathcal{K}_{\chi}$ as a direct summand, up to isomorphism. This claim follows by combining
Proposition~\ref{prop12} with Theorem~\ref{theorem:classification_sch}.
\end{proof}

We note that, by Proposition~\ref{prop12}, the combinatorics
of the $\cF$-module category $\mathcal{K}_{\chi}$ does not depend on $\chi$.
Furthermore, by comparing Propositions~\ref{proposition:V_otimes_L} and 
\ref{prop12}, we see that the combinatorics of 
the $\cF$-module category $\mathcal{H}_{\chi}$ is different from 
the combinatorics of the $\cF$-module category $\mathcal{K}_{\chi}$.

\section{Some general results on $\g$-Harish-Chandra modules}
\label{section:existence_general}

\subsection{A sufficient condition for existence of simple infinite-dimensional $\g$-Harish-Chandra modules}

Recall our general setup from Section \ref{s2}, where $\LL \cong \g \ltimes \rr$ was arbitrary finite-dimensional Lie algebra. Assume that a triangular decomposition (\ref{equation:general_triangular}) is fixed. Denote by $\rr_0$ the zero-weight space of $\rr$. Obviously, we have $\tilde{\h} = \h \oplus \rr_0$.

To prove the main theorem in this section, we need to use another variant of Enright's and Arkhipov's functors. It will be the same as $\EnAr$ from before, but without taking the locally finite for the positive root vector. Fix a simple reflection $s$, and the corresponding $\mathfrak{sl}_2$-triple $\{f,h,e\} \subseteq \g$. For an $\LL$-module $M$, set
\[ \Com_s(M) := S_f \tens{U(\LL)} M  , \]
where, as before, $S_f$ denotes the localized algebra $U(\LL)_{(f)}$ modulo $U(\LL)$. This functor commutes with the forgetful functor that forgets the $\rr$-action. Moreover, it preserves $\g$-central characters. 

\begin{remark} \label{remark:twisting}
Note that if we would twist the action on $\Com_s(M)$ by the inner automorphism of $\LL$ that corresponds to the simple reflection $s$, we would get exactly the twisting functor $T_s(M)$ from \cite{arkhipov2004algebraic, andersen2003twisting}. We conclude that if $M$ is from the category $\mathcal{O}$ for $\g$, then $\Com_s(M)$ is also from the category $\mathcal{O}$ for $\g$, but for another choice of Borel subalgebra.
\end{remark}

\begin{theorem} \label{theorem:exist_inf_dim_HC}
Suppose that $[\LL,\rr] \cap \rr_0 \neq 0$. Then there exists an infinite-dimensional simple $\g$-Harish-Chandra module for $\LL$.
\end{theorem}
\begin{proof}
Fix an element $c \in [\LL,\rr] \cap \rr_0$. Then $c \in \Nrad(\LL)$, so $c$ must annihilate any simple finite-dimensional $\LL$-module.

The idea is to start with the simple quotient of a Verma module for $\LL$, which has finite multiplicities of its $\g$-submodules (but possibly infinite-dimensional), and then use the functors $\Com_s$ to obtain an $\LL$-module that should have a finite-dimensional $\g$-submodule. Then the image of the universal $\LL$-module in the constructed module should have only finite-dimensional $\g$-submodules with finite multiplicities. The element $c$ will insure infinite-dimensionality.

Fix an antidominant, regular and integral $\lambda \in \h^\ast$ and extend it to $\tilde{\lambda} \in \tilde{\h}^\ast$ such that $\tilde{\lambda}(c) \neq 0$. Consider the Verma module $\Delta(\tilde{\lambda})$ as in (\ref{equation:Verma_def}), and its simple quotient $\mathbf{L}(\tilde{\lambda})$.

By considering the $\h$-weight spaces of $\Delta(\tilde{\lambda})$ (Proposition \ref{proposition:Vermas_general}), it follows that as a $\g$-module, $\mathbf{L}(\tilde{\lambda})$ has a $\g$-direct summand $\Delta^\g(\lambda)$ generated by the highest weight vector $\mathbf{v} \in \mathbf{L}(\tilde{\lambda})$, and this is the only $\g$-composition factor of $\mathbf{L}(\tilde{\lambda})$ of this $\g$-central character.

From the fact that $\Delta(\tilde{\lambda})$ and $\mathbf{L}(\tilde{\lambda})$ have finite-dimensional weight spaces, it follows that, when considered as $\g$-modules, they contain any simple $\g$-composition factor with at most finite multiplicity. Moreover, from this it follows that their components in any fixed $\g$-central character lie in the category $\mathcal{O}$ for $\g$.

The element $c$ acts on $\mathbf{v} \in \Delta^\g(\lambda) \subseteq \mathbf{L}(\tilde{\lambda})$ by the scalar $\tilde{\lambda}(c) \neq 0$.

Choose a reduced expression $w_0 = s_1 s_2 \ldots s_k$ of the longest element in the Weyl group for $\g$, and set $\Com_{w_0} = \Com_{s_1} \circ  \Com_{s_2} \circ \ldots \circ \Com_{s_k}$. From Remark \ref{remark:twisting} and \cite[Theorem 2.3]{andersen2003twisting} it follows that $\Com_{w_0}(\Delta^\g(\lambda))$ is isomorphic to the dual dominant $\g$-Verma module for the opposite Borel subalgebra. From this, we can conclude $\Com_{w_0}(\Delta^\g(\lambda))$ contains the finite-dimensional $\g$-submodule $L(\mu)$, where $\mu := w_0 \cdot \lambda$  is dominant, integral and regular. Moreover, by observing what is happening on the $\mathfrak{sl}_2$-subalgebras of $\g$, one can conclude that the lowest weight vector in $L(\mu)$ is given by
\begin{equation} \label{equation:Com_v}
f_1^{-1} f_2^{-1} \ldots f_k^{-1} \mathbf{v},
\end{equation}
where $f_i$ is the negative root vector corresponding to $s_i$.

It follows that $M:=\Com_{w_0}(\mathbf{L}(\tilde{\lambda}))$ as a $\g$-module has also $L(\mu)$ as a $\g$-direct summand. Moreover, from Remark \ref{remark:twisting} it follows that $L(\mu)$ appears in $M$ precisely once, and that each simple $\g$-module appears with at most finite multiplicity.

Consider $Q(\mu) := \Ind_{\g}^{\LL}(L(\mu)) = U(\LL) \tens{U(\g)} L(\mu)$. It has only finite-dimensional $\g$-composition factors, but possibly with infinite multiplicities.

By the universal property of the induction functor, we get a non-zero $\LL$-homomorphism $\varphi \colon Q(\mu) \to M$, hitting the $\g$-submodule $L(\mu)$ in $M$. Denote by $N$ the image of this map. It follows from the construction $N$ is a $\g$-Harish-Chandra module, generated by its unique occurence of the $\g$-type $L(\mu)$.

Furthermore, $N$ has a unique simple quotient $V$, which contains this $\g$-type $L(\mu)$. Clearly, $V$ is a simple $\g$-Harish-Chandra $\LL$-module.

But also, $V$ is infinite-dimensional. To see this, it is enough to check that $c$ does not annihilate the vector (\ref{equation:Com_v}). Observe that
\begin{align*}
c \cdot f_1^{-1} f_2^{-1} \ldots f_k^{-1} \mathbf{v} &= \tilde{\lambda}(c) \cdot f_1^{-1} f_2^{-1} \ldots f_k^{-1} \mathbf{v} + \sum_{i=1}^k f_1^{-1} \ldots [c,f_i^{-1}] \ldots f_k^{-1} \mathbf{v}.
\end{align*}
Since $[c,f_i^{-1}] = f_i^{-1}[f_i,c]f_i^{-1}$, each summand with $[c,f_i^{-1}]\neq 0$ will contain a factor from $\Nrad(\LL)$. Therefore, these terms cannot cancel with $\tilde{\lambda}(c) \cdot f_1^{-1} f_2^{-1} \ldots f_k^{-1} \mathbf{v}$, and so the total result is non-zero.
\end{proof}

We believe that the connection to the highest weight theory could be established in a more general setup, at least for the Takiff Lie algebras. Using the notation from Section \ref{section:Takiff}, we formulate:

\begin{conjecture} \label{conjecture:EA_Q0}
Let $\T$ be a Takiff Lie algebra attached to a semi-simple Lie algebra $\g$. For a ``generic'' $\chi \colon \overline{Z(\g)} \to \C$, there is $\lambda \in \bar{\h}^\ast$ such that $\EnAr_{w_0}(\Delta(-2\rho,\lambda)) \cong Q(0,\chi)$.
\end{conjecture}

Here $\rho$ is the half-sum of all elements in $\Delta^+(\g,\h)$, and $\EnAr_{w_0}$ should be defined as the composition $\EnAr_{s_1} \circ \EnAr_{s_2} \circ \ldots \circ \EnAr_{s_k}$, where $w_0 = s_1 s_2 \ldots s_k$ is a fixed a reduced expression of the longest element in the Weyl group for $\g$. Each $\EnAr_s$ should be defined as in Definition \ref{definition:EA}, i.e., as $\EnAr_s := \prescript{e}{}{\left(\Com_s(-)\right)}$ for the $\mathfrak{sl}_2$-triple $\{f,h,e\}$ corresponding to $s$. It is not a priori clear that such $\EnAr_{w_0}$ does not depend on the choice of a reduced expression.

Conjecture \ref{conjecture:EA_Q0} was already proved for the Takiff $\mathfrak{sl}_2$ case in Section \ref{section:takiff_sl2}. Analogous statement is proved also for the Schr\"{o}dinger Lie algebra in Section \ref{section:schrodinger}.

\subsection{On classification of simple $\g$-Harish-Chandra modules for generalized Takiff Lie algebras}

In this subsection we consider
a finite-dimensional Lie algebra $\LL$ with a fixed Levi decomposition 
$\LL \cong \g \ltimes \rr$ and assume that:
\[ \rr \text{ is abelian}. \]
It is reasonable to call such algebras \textit{generalized Takiff Lie algebras}.
In analogy to Subsection~\ref{spureT},
we consider the {\em purely radical} part $\overline{Z(\LL)}:=Z(\LL)\cap U(\rr)$
of the center $Z(\LL)$ of $U(\LL)$. Since $\rr$ is assumed to be abelian, it is obvious that $\overline{Z(\LL)} = U(\rr)^\g$, the $\g$-invariants in $U(\rr) \cong \Sym(\rr)$ with respect to the adjoint action. For brevity, algebra homomorphisms 
$\chi:\overline{Z(\LL)}\to \mathbb{C}$ will be loosely called {\em radical 
central characters}.

For $\lambda \in \h^\ast$ integral, dominant and regular, we have the  \emph{universal module} 
\begin{equation*}
Q(\lambda) := \Ind_{\g}^{\LL} L(\lambda) = U(\LL) \tens{U(\g)} L(\lambda) \cong U(\rr) \tens{\C} L(\lambda) \cong \Sym(\rr) \tens{\C} L(\lambda).
\end{equation*}
Clearly, it is $\g$-locally finite. Completely analogously to Proposition \ref{proposition:end_Q0}, one can show that $\End Q(0) \cong \overline{Z(\LL)}$. Given also a radical central character $\chi \colon \overline{Z(\LL)} \to \C$, we define
\begin{equation*}
Q(\lambda,\chi):=\quotient{Q(\lambda)}{\mathbf{m}_\chi Q(\lambda)},
\end{equation*}
where $\mathbf{m}_\chi :=  \Ker \chi$ is the maximal ideal in $\overline{Z(\LL)}$ corresponding to $\chi$.

From \cite[Proposition 6.5]{knapp1988lie} we have $Q(\lambda) \cong Q(0) \otimes L(\lambda)$, and from this and right-exactness of tensor product we conclude $Q(\lambda,\chi) \cong Q(0,\chi) \otimes L(\lambda)$. In these formulas $L(\lambda)$ is considered to be an $\LL$-module with the trivial action of $\rr$.

By construction, the action of $\g$ on $Q(0,\chi)$ is locally finite, with the trivial
module $L(0)$ having multiplicity exactly $1$ in $Q(0,\chi)$. Since $Q(0,\chi)$ is generated
by this unique copy of $L(0)$, it follows that $Q(0,\chi)$ has a unique simple quotient, which we denote by $V(0,\chi)$.

From Theorem \ref{theorem:Q_simple} we know that $Q(0,\chi)=V(0,\chi)$ when $\LL$ is the Takiff $\mathfrak{sl}_2$ and $\chi$ is not the radical central character of the trivial module, but generally this fails, for example when $\rr \cong L(4)$ as a $\g$-module, as we will see in Subsection \ref{subsec7.3}.

We will use the notion of Gelfand-Kirillov dimension, for the definition and basic properties see \cite[2.3]{CCM} or \cite{krause2000growth}. The Gelfand-Kirillov dimension of a finitely generated module $M$ will be denoted by $\GK(M)$ and  the Bernstein number by $\e(M)$.

\begin{proposition} \label{propnew5}
Fix a radical central character $\chi$.
\begin{itemize}
\item $V(0,\chi)$ is the unique simple $\LL$-module having both radical central character $\chi$ and the trivial $\g$-module as one of $\g$-types.

\item If $[\g,\rr]=\rr$, then $V(0,\chi)$ is either the trivial $\LL$-module (precisely when $\chi$ is the radical central character of the trivial $\LL$-module), or infinite-dimensional.

\item $V(0,\chi)$ is a $\g$-Harish-Chandra module.
\end{itemize} 
\end{proposition}

\begin{proof}
The first claim follows from universal property of the induction functor.

For the second claim, recall that $[\g,\rr]=\Nrad \LL$ is the intersection of all kernels of simple finite-dimensional $\LL$-modules (see \ref{s2}). Therefore any simple finite-dimensional $\LL$-module necessarily has the same radical central character as the trivial module $\LL$-module. From the uniqueness in the first part it follows that $V(0,\chi)$ is either trivial or infinite-dimensional.

It remains to prove that $V(0,\chi)$ is a $\g$-Harish-Chandra module. It is $\g$-locally finite by construction. Assume that $[V(0,\chi) \colon L] = \infty$ for some simple finite-dimensional $\g$-module $L$. The module $V(0,\chi) \otimes L^\ast$ is also $\g$-locally finite and finitely generated. By \cite[Lemma 8.8]{jantzen11983einhullende} we have
\begin{equation} \label{equation:GK_tensor}
\GK(V(0,\chi) \otimes L^\ast) = \GK(V(0,\chi)) \ \text{ and } \ \e(V(0,\chi) \otimes L^\ast) = \dim L \cdot \e(V(0,\chi)).
\end{equation}
Moreover,
\begin{multline*}
[ V(0,\chi) \otimes L^\ast \colon L(0) ] = \Hom_\g \left(L(0),V(0,\chi) \otimes L^\ast\right)) = \\
= \Hom_\g \left(L(0)\otimes L ,V(0,\chi)\right) = [V(0,\chi) \colon L] = \infty. \end{multline*}
By uniqueness in the first part, it follows that $V(0,\chi)$ appears in $V(0,\chi) \otimes L^\ast$ infinitely many times. This is a contradiction with (\ref{equation:GK_tensor}) and \cite[Lemma 8]{CCM}.
\end{proof}

\begin{remark}
One can check that Gelfand-Kirillov dimensions of all infinite-dimensional simple $\g$-Harish-Chandra modules for the Takiff $\mathfrak{sl}_2$ (Section \ref{section:takiff_sl2}) as well for the Schr\"{o}dinger Lie algebra (Section \ref{section:schrodinger}) are $2$.
\end{remark}

\begin{proposition}\label{propnew6}
Let $\chi$ be a radical central character and $L$ a simple finite-dimensional $\g$-module. Then $\LL$ has a simple $\g$-Harish-Chandra module with the radical
central character $\chi$ and having  $L$ as one of $\g$-types.

Moreover, the number of (isomorphism classes of) simple $\g$-Harish-Chandra modules with the radical
central character $\chi$ and having $L$ as one of its $\g$-types is at most
\begin{equation}\label{eqnew7}
\sum_{M} l_{M,L}^L [V(0,\chi):M] < \infty,
\end{equation}
where $M$ runs through the set of isomorphism classes of simple finite-dimensional
$\g$-modules and $l_{M,L}^L$ is the Littlewood-Richardson coefficient, i.e. the multiplicity
of $L$ in $M\otimes L$.
\end{proposition}

We note that only finitely many summands in \eqref{eqnew7} are non-zero,
as $l_{M,V}^V\neq 0$ if and only if $M$ appears as a summand of $V\otimes V^{*}$.
\begin{proof}
Since we know from Proposition \ref{propnew5} that $V(0,\chi)$ is a $\g$-Harish-Chandra module having $L(0)$ as a $\g$-type, it follows that  $L\otimes V(0,\chi)$ is also a $\g$-Harish-Chandra module, having $L$ as one of $\g$-types. This implies the first part of the proposition.

Conversely, let $N$ be a simple  $\g$-Harish-Chandra module with the radical
central character $\chi$ and having  $L$ as one of $\g$-types.
Then $L^{*} \otimes N$ is a $\g$-Harish-Chandra module with the radical
central character $\chi$ and having  $L(0)$ as one of $\g$-types, and therefore contains $V(0,\chi)$ as a subquotient.

The above arguments, combined with the biadjunction $(L\otimes{}-,L^{*}\otimes{}-)$, 
imply that any simple  $\g$-Harish-Chandra module with the radical
central character $\chi$ and having  $L$ as one of $\g$-types is a subquotient of
$L \otimes V(0,\chi)$. The latter is a $\g$-Harish-Chandra module in which the
multiplicity of $L$ is bounded by the expression in \eqref{eqnew7} (see the proof of Proposition \ref{proposition:multiplicity}).
This implies the second part of the proposition.
\end{proof}

\section{On $\mathfrak{sl}_2$-Harish-Chandra modules for other conformal Galilei algebras} \label{section:conf_galilei}

\subsection{Conformal Galilei algebras}

By a \emph{conformal Galilei algebra} we mean a semi-direct product 
$\LL^n := \mathfrak{sl}_2 \ltimes L(n)$,  where $n \in \Z_{\geq 0}$. 
Here $L(n)$ is an abelian ideal on which $\g := \mathfrak{sl}_2$ acts in the obvious way. 
For a more general definition and various central extensions, 
see \cite{lu2014simple, alshammari2019on, gomis2012schrodinger}.

Note that we have $\LL^0 \cong \mathfrak{gl}_2$, $\LL^1$ is the centerless Schr\"{o}dinger Lie algebra, and $\LL^2$ is the Takiff $\mathfrak{sl}_2$.

We have $\Nrad(\LL^n) = \rr = L(n)$, and so $[\LL^n,\rr] \cap \rr_0 \cong \C$ if and only if $n$ is even (i.e., $L(n)$ is odd-dimensional), otherwise $[\LL^n,\rr] \cap \rr_0 = 0$.

Denote by $v_n, v_{n-2}, \ldots, v_{-n}$ a basis of $L(n)$ such that each $v_i$ is a weight vector of weight $i$, and $[e,v_{n-2i}] = (n-i+1) v_{n-2i+2}$, for $i \in \{1,2,\ldots,n\}$.

\subsection{The Lie algebra $\LL^3 = \mathfrak{sl}_2 \ltimes L(3)$}\label{subscounter}

Since $\rr = L(3)$ has trivial zero-weight space, the assumption in Theorem \ref{theorem:exist_inf_dim_HC} is not satisfied. Nevertheless, we will show that simple infinite-dimensional $\g$-Harish-Chandra modules exist. So the converse of Theorem \ref{theorem:exist_inf_dim_HC} is not true. This suggest that the highest weight theory is not enough to obtain and classify $\g$-Harish-Chandra modules for any Lie algebra.

From the classical invariant theory it is well known that $\Sym(\rr)^\g$ is generated by only one element $C$, homogeneous of degree $4$, the so called cubic discriminant (see e.g. \cite[Lecture XVII]{hilbert1993theory}): 
\[ C= v_{-1}^2 v_{1}^2 - 27 v_{-3}^2 v_{3}^2 - 4 v_{-1}^3v_3 - 4 v_{-1}v_3^3 + 18v_{-3}v_{-1}v_1 v_3. \]
This expression of $C$ is just for the record, we will not use it in the computations. We identify radical central characters with their value on $C$.

For $\chi \in \C$ we have the universal $\LL^3$-module $Q(0,\chi) = \quotient{Q(0)}{(C-\chi)}$.

\begin{proposition}
\label{proposition:Q0_for_L3}
For $\chi \neq 0$, the module $Q(0,\chi)$ is $\g$-Harish-Chandra. The multiplicities of its $\g$-types are given by $[ Q(0,\chi) \colon L(k)] = k - 2 \left\lfloor{\frac{k+2}{3}}\right\rfloor+1$.
\end{proposition}
\begin{proof}
From the main result of \cite{futorny2005kostant} it follows easily that $\Sym(\rr)$ is free as a module over $\Sym(\rr)^\g = \C[C]$. This implies that the $\g$-structure of $Q(0,\chi)$ does not depend on the choice of $\chi$. So in particular, it is enough to prove the finite-multiplicity statement for $Q(0,0)$, which is as a $\g$-module isomorphic to $\quotient{\Sym(\rr)}{C \cdot \Sym(\rr)}$.

But these are now graded modules, so we can subtract their graded characters. More precisely, for $d \geq 4$ and $k \geq 0$ we have
\begin{equation} \label{equation:symm_character}
[Q(0,0)^d \colon L(k)] = [\Sym^d(\rr) \colon L(k)] - [\Sym^{d-4}(\rr) \colon L(k)], \end{equation}
where $(-)^d$ denotes the homogeneous part of degree $d$.

Using \cite{hahn2018from}, one can calculate the right-hand side of (\ref{equation:symm_character}). For non-negative integers $a,b,c$ let $p(a,b,c)$ denote the number of partitions of $c$ into at most $b$ parts, and each part bounded above by $a$. Denote $N(a,b,c) := p(a,b,c) - p(a,b,c-1)$ if $c \geq 1$, and set $N(a,b,0) := 1$. By \cite[Theorem 3.1]{hahn2018from}, the multiplicity (\ref{equation:symm_character}) is equal to $0$ if $k \not\in \{3d,3d-2,3d-4,\ldots\}$, and to
\[ N \left(d,3,\frac{3d-k}{2}\right) - N\left(d-4,3,\frac{3d-k}{2}-6\right) \]
otherwise. But the latter is also equal to $0$ whenever $k<d-4$, by using the formulas in \cite[Corollary 3.2]{hahn2018from}. It follows that any $L(k)$ can appear in $Q(0,0)$ in at most degree $k+4$, hence only finitely many times.

Using the same formulas in \cite[Corollary 3.2]{hahn2018from}, one can derive the multiplicity formula from the statement of the proposition. We omit the details.
\end{proof}

Either from Proposition \ref{propnew5}, or from Proposition \ref{proposition:Q0_for_L3}, we have:

\begin{corollary}
The unique simple quotient $V(0,\chi)$ of $Q(0,\chi)$ is $\g$-Harish-Chandra. It is infinite-dimensional if $\chi \neq 0$.
\end{corollary}

We do not know whether $Q(0,\chi)=V(0,\chi)$, i.e., whether $Q(0,\chi)$ is already simple (for $\chi \neq 0$), as was in the Takiff $\mathfrak{sl}_2$ and the Schr\"{o}dinger cases.

\subsection{The Lie algebra $\LL^4 = \mathfrak{sl}_2 \ltimes L(4)$}\label{subsec7.3}

Consider now the algebra $\LL^4$. In this subsection we classify simple 
$\g$-Harish-Chandra modules for $\LL^4$ which appear in Enright-Arkhipov
completions of simple highest weight modules.

The algebra $\Sym(\rr)^\g$ is generated by two algebraically independent elements, homogeneous of degrees $2$ and $3$ (see e.g. \cite[Lecture XVIII]{hilbert1993theory}):
\begin{align*}
C_2 &= v_0^2 -3 v_{-2} v_2 + 12 v_{-4} v_4, \\
C_3 &= v_0^3 - \frac{9}{2} v_{-2} v_0 v_2 + \frac{27}{2} v_{-2}^2 v_4  + \frac{27}{2} v_{-4} v_2^2  - 36 v_{-4} v_0 v_4.
\end{align*}

Recall also, from \cite[Theorem~4]{lu2014simple}, the structure of simple
highest weight $\LL^4$-modules. Let $\tilde{\h}$ denote the (generalized) Cartan subalgebra 
of $\LL^4$ spanned by $h$ and $v_0$. If $\lambda\in\tilde{\h}^\ast$ is such that 
$\lambda(v_0)=0$, then $\mathfrak{r}$ annihilates the corresponding simple
highest weight module $\mathbf{L}(\lambda)$. If $\lambda\in\tilde{\h}^\ast$ is such that 
$\lambda(v_0)\neq0$, then the restriction of $\mathbf{L}(\lambda)$ to $\g$ has a 
multiplicity free Verma filtration with subquotients of the form 
$\Delta^{\g}(\lambda-n\alpha)$, where $n\in\mathbb{Z}_{\geq 0}$ and $\alpha$
is the root corresponding to $e\in\g$. Note that the elements
$C_2$ and $C_3$ act on $\mathbf{L}(\lambda)$ as the scalars $\lambda(v_0)^2$
and $\lambda(v_0)^3$, respectively.

Denote by $\mathcal{F}$ the semi-simple additive category generated
by simple subquotients of Enright-Arkhipov completions of simple
highest weight $\LL^4$-modules. Note that all modules in $\mathcal{F}$
are $\g$-Harish-Chandra modules for $\LL^4$.
Our main result of this subsection is the following theorem.

\begin{theorem}\label{thmconf4new}
\begin{enumerate}[$($a$)$]
\item\label{thmconf4new.1} For each $\tilde{\lambda}\in\mathbb{C}\setminus\{0\}$
and for each $i\in\Z_{>0}$, there is a unique, up to isomorphism,
simple object $V(i,\tilde{\lambda})$ in $\mathcal{F}$ on which $C_j$, where $j=2,3$,
acts via $\tilde{\lambda}^j$ and which has $\g$-types $L(i)$, $L(i+2)$, $L(i+4),\dots$,
all multiplicity free. 
\item\label{thmconf4new.2} For each $\tilde{\lambda}\in\mathbb{C}\setminus\{0\}$, 
there is a unique, up to isomorphism, simple object $V'(0,\tilde{\lambda})$ in 
$\mathcal{F}$ on which $C_j$, where $j=2,3$, acts via $\tilde{\lambda}^j$ 
and which has $\g$-types $L(0)$, $L(4)$, $L(8),\dots$,
all multiplicity free. 
\item\label{thmconf4new.3} For each $\tilde{\lambda}\in\mathbb{C}\setminus\{0\}$, 
there is a unique, up to isomorphism, simple object $V'(2,\tilde{\lambda})$ in 
$\mathcal{F}$ on which $C_j$, where $j=2,3$, acts via $\tilde{\lambda}^j$ 
and which has $\g$-types $L(2)$, $L(6)$, $L(10),\dots$,
all multiplicity free. 
\item\label{thmconf4new.5} The modules above provide a complete and irredundant
list of representatives of isomorphism classes of simple objects in $\mathcal{F}$.

\item\label{thmconf4new.6} Let $V$ be a simple $\g$-Harish-Chandra module on which $C_j$, where $j=2,3$, acts via $\tilde{\lambda}^j$, for some $\tilde{\lambda}\in\mathbb{C}\setminus\{0\}$. Then $V$ belongs to $\mathcal{F}$.
\end{enumerate}
\end{theorem}

To prove this result, we will need some preparation.
The following lemma extends \cite[Corollary~3.5]{hahn2018from}
(note that the case treated in the lemma below is referred to as 
``complicated'' in \cite{hahn2018from}).

\begin{lemma}\label{lemlemnew754}
For every non-negative integer $k$, we have
\begin{displaymath}
\sum_{s=0}^{\lfloor\frac{2k-1}{4}\rfloor}
\left(
\left\lfloor\frac{2k-4s-1}{2}\right\rfloor-
\left\lfloor\frac{2k-4s-2}{3}\right\rfloor
\right)
-\sum_{s=0}^{k-2}
\left(
\left\lfloor\frac{s}{2}\right\rfloor-
\left\lfloor\frac{s-1}{3}\right\rfloor
\right)
=0. 
\end{displaymath}
\end{lemma}

\begin{proof}
Using computer, it is easy to check that the claim of the lemma  
is true for small values of $k$ (we checked this independently 
and by different methods using Scilab and SageMath up to $k=200$).
After that, one can do induction on $k$ with induction step $12$.
So, we write $k=12a+r$ and consider each $r$ separately. Let
$S(k)$ denote the left hand side of the formula. For $k>12$,
the value $S(k)-S(k-12)$ can be written as a polynomial in $a$
(the polynomial itself depends on $r$) of degree at most two.

From the original computation it follows that
$S(k)-S(k-12)$ vanishes for enough values of $a$ to conclude that
$S(k)-S(k-12)$ is identically $0$. The claim follows.
\end{proof}

\begin{remark}
The results of \cite{hahn2018from} say that 
Lemma~\ref{lemlemnew754} is equivalent to the fact that, for each $k\geq 0$, the set $\Lambda_1$ of all vectors $(a,b,c,d,e)$ with non-negative integer coefficients satisfying $a+b+c+d+e=k$ and
$2a+b-d-2e=1$ has the same cardinality as the set $\Lambda_2$ of all vectors $(a,b,c,d,e)$ with non-negative integer coefficients satisfying $a+b+c+d+e=k$ and $2a+b-d-2e=2$.
We give here explicitely a bijection between these sets.
First, note that $\{ (a,b,c,d,e) \in \Lambda_1 \colon e \neq 0 \}$ maps bijectively to $\{ (a',b',c',d',e') \in \Lambda_2 \colon d' \neq 0 \}$ by \[(a,b,c,d,e)\mapsto (a,b,c,d+1,e-1). \]
%
%
%
%
The remainder $\{ (a,b,c,d,0) \in \Lambda_1\}$ maps bijectively to $\{ (a',b',c',0,e') \in \Lambda_2\}$ by the formula
\[ \left( a,b, c, d,0  \right) \mapsto \left( b, 2a, c, 0, \frac{b+d-1}{2} \right).  \]
\end{remark}

\begin{lemma}\label{lemlemnew9423w}
Let $V$ be a simple infinite-dimensional $\g$-Harish-Chandra module
for $\LL^4$. Then each $v_i$ acts injectively on $V$.
\end{lemma}

\begin{proof}
Since the adjoint action of $v_i$ on $\LL^4$ is locally nilpotent, 
the action of $v_i$ on each simple $\LL^4$-module is either injective
or locally nilpotent, cf. \cite[Section~3]{DMP}. 
Note that the action of both $e$ and $f$ on
$V$ is locally nilpotent by definition. Let $x\in V$ be such that
$v_i\cdot x=0$, for some $i\neq 4$, and $e^m\cdot v=0$. 
Then $\mathrm{ad}^m(v_i)(e^m)\cdot x=0$ and it is easy to check that 
this implies that $v_{i+2}^m\cdot x=0$. That is, the action of
$v_{i+2}$ is locally nilpotent. Applying similar arguments
using $e$ and $f$, we get that the action of all $v_{j}$'s is
locally nilpotent.

As all $v_{j}$'s commute, $V$ must contain some non-zero $x$
which is killed by all the $v_{j}$'s. Since the adjoint action of
$e$ preserves the set of the $v_{j}$'s, we can even assume that
$e$ kills $x$. But then this means that $V$ is a highest weight
module. Being also a $\g$-Harish-Chandra module, this implies that
$V$ must be finite-dimensional, a contradiction.
\end{proof}

For $\chi_2, \chi_3 \in \C$, we recall the universal $\LL^4$-module
\[ Q(0,\chi_2, \chi_3) = \quotient{Q(0)}{(C_2-\chi_2,C_3-\chi_3)}, \]
and its unique simple quotient $V(0,\chi_2, \chi_3)$ containing $L(0)$.

\begin{proposition}\label{propnew82734y2}
Fix $\tilde{\lambda} \in \C\setminus\{0\}$. The module $V(0,\tilde{\lambda}^2, \tilde{\lambda}^3)$ is a simple infinite-dimensional $\g$-Harish-Chandra module. Its $\g$-types are $L(0), L(4), L(8) \ldots$, each occurring with multiplicity one.
\end{proposition}
\begin{proof}
By \cite[Theorem  3.1 and Corollary~3.4]{hahn2018from}, the multiplicity of $L(2)$
in $Q(0)$ is given by the left hand side of the formula from Lemma~\ref{lemlemnew754}.
Therefore, by Lemma~\ref{lemlemnew754}, $L(2)$ does not appear in $Q(0)$.

Now, take $\lambda$ such that $\lambda(h)=-2$ and $\lambda(v_0)=\tilde{\lambda}$.
Then $\mathbf{EA}(\mathbf{L}(\lambda))$ is a $\g$-Harish-Chandra module having multiplicity-free 
$\g$-types $L(0)$, $L(2)$, $L(4),\dots$. As $v_4$ commutes with $e$, from
Lemma~\ref{lemlemnew9423w} it follows that $v_4$ sends each non-zero highest weight 
vector of $L(i)$ inside $\mathbf{EA}(\mathbf{L}(\lambda))$ to a non-zero highest weight vector of
$L(i+4)$ inside $\mathbf{EA}(\mathbf{L}(\lambda))$. Note that all simple subquotients of
$\mathbf{EA}(\mathbf{L}(\lambda))$ must be infinite-dimensional as the central characters of
$\mathbf{EA}(\mathbf{L}(\lambda))$ is different, by construction, from the central characters
of simple finite-dimensional $\LL^4$-modules.

By the universal property of $Q(0)$, the inclusion of $L(0)$ in $\mathbf{EA}(\mathbf{L}(\lambda))$
gives rise to a non-zero homomorphism from $Q(0)$ to $\mathbf{EA}(\mathbf{L}(\lambda))$.
The image $V$ of this homomorphism does not contain $L(2)$, as was established in the
first paragraph of the proof. Therefore, from Lemma~\ref{lemlemnew9423w} it follows
that $V$ has $\g$-types $L(0)$, $L(4)$, $L(8),\dots$ and the quotient
$\mathbf{EA}(\mathbf{L}(\lambda))/V$ has $\g$-types $L(2)$, $L(6)$, $L(10),\dots$.
In fact, from Lemma~\ref{lemlemnew9423w} and the above remark that  
all simple subquotients of
$\mathbf{EA}(\mathbf{L}(\lambda))$ must be infinite-dimensional, it follows that both
$V$ and $\mathbf{EA}(\mathbf{L}(\lambda))/V$ are simple modules.

This implies that $V\cong V(0,\tilde{\lambda}^2, \tilde{\lambda}^3)$ and the
claim of the lemma follows.
\end{proof}

\begin{lemma}\label{lemnew1403985y}
If $\lambda(v_0)\neq 0$, then, in the category of $h$-weight $\LL^4$-modules,
we have the vanishing $\mathrm{Ext}^1(\mathbf{L}(\lambda-\alpha),\mathbf{L}(\lambda))=0$.
\end{lemma}

\begin{proof}
Let $\mathbf{L}(\lambda)\hookrightarrow M\twoheadrightarrow \mathbf{L}(\lambda-\alpha)$ be a short
exact sequence in the category of $h$-weight $\LL^4$-modules. Consider the vector space 
$X:=M_{\lambda}\oplus M_{\lambda-\alpha}$ and note that it is killed by $v_4$.
Therefore this vector space is a module over the polynomial algebra $A$ in
$e$ and $v_2$. The space of first self-extensions for each simple $A$-module
is two-dimensional. Since $\mathbf{L}(\lambda)$ is simple, the submodule 
$Y:=\mathbf{L}(\lambda)_{\lambda}\oplus \mathbf{L}(\lambda)_{\lambda-\alpha}$ of $X$ is indecomposable.
Since $\mathbf{L}(\lambda)_{\lambda-\alpha}$ has dimension $2$, $Y$ is the universal self-extension
of the trivial $A$-module. This implies that, in the category of $A$-modules, 
the first extension from  $\mathbf{L}(\lambda-\alpha)_{\lambda-\alpha}$ to $Y$ coming 
from the socle of $Y$ vanishes. Consequently,
$M$ must have a non-zero vector of weight $\lambda-\alpha$ which is killed by both
$e$ and $v_2$. As the adjoint action of $v_0$ leaves the span of $e$ and $v_2$
invariant, it follows that $M$ contains a highest weight vector of weight
$\lambda-\alpha$. Consequently, $M$ splits, proving the claim.
\end{proof}

Now we are ready to prove Theorem~\ref{thmconf4new}.

\begin{proof}[Proof of Theorem~\ref{thmconf4new}]
We take $\lambda$ such that $\lambda(h)=-2$ and $\lambda(v_0)=\tilde{\lambda}$.

The $\LL^4$-module $V'(0,\tilde{\lambda}):=V(0,\tilde{\lambda}^2, \tilde{\lambda}^3)$
and the module $V'(2,\tilde{\lambda}):=\mathbf{EA}(L(\lambda))/V$, cf. the proof of
Proposition~\ref{propnew82734y2}, are already constructed. Note that the proof of
Proposition~\ref{propnew82734y2} implies 
\begin{displaymath}
\mathrm{Ext}^1(V'(0,\tilde{\lambda}),V'(2,\tilde{\lambda})) =0
\end{displaymath}
in the category of $\g$-Harish-Chandra modules. As usual, on the category of 
$\g$-Harish-Chandra modules we have the restricted duality, which we denote by 
$\star$, that maps $\displaystyle \bigoplus_i L(i)^{\oplus m_i}$ to 
$\displaystyle\bigoplus_i (L(i)^*)^{\oplus m_i}$.
The fact that $V'(0,\tilde{\lambda})$ is self-dual follows directly from 
its uniqueness given by construction. Applying $\star$, we obtain
\begin{displaymath}
\mathrm{Ext}^1(V'(2,\tilde{\lambda})^{\star},V'(0,\tilde{\lambda})) =0,
\end{displaymath}
where the modules $V'(2,\tilde{\lambda})$ and $V'(2,\tilde{\lambda})^{\star}$
have the same $\g$-types. If we assume that 
$V'(2,\tilde{\lambda})\not\cong V'(2,\tilde{\lambda})^{\star}$,
then, from Proposition~\ref{propnew6}, it follows that these are the only 
simple $\g$-Harish-Chandra modules having $L(2)$ as a $\g$-type. 

Consider now the module $L(2)\otimes V'(0,\tilde{\lambda})$. By adjunction, this
must have both $V'(2,\tilde{\lambda})$ and $V'(2,\tilde{\lambda})^{\star}$ as simple
subquotients, both with multiplicity one. 
The remaining $\g$-types are $L(4)$, $L(8)$, $L(10).\dots$.
If $V$ is a subquotient of $L(2)\otimes V'(0,\tilde{\lambda})$ whose $\g$-types
form a subset of these remaining $\g$-types, then $V$, 
by adjunction, cannot be in the top or socle of $L(2)\otimes V'(0,\tilde{\lambda})$
as $L(2)\otimes V$ does not have $L(0)$ as a simple subquotient.
This implies that $L(2)\otimes V'(0,\tilde{\lambda})$ must have socle
$V'(2,\tilde{\lambda})$ and top $V'(2,\tilde{\lambda})^{\star}$ or vice-versa.
However, since both $V'(0,\tilde{\lambda})$ and $L(2)$ are self-dual, 
so is $L(2)\otimes V'(0,\tilde{\lambda})$, a contradiction.
Therefore $V'(2,\tilde{\lambda})\cong V'(2,\tilde{\lambda})^{\star}$ and thus
\begin{equation} \label{equation:EA=sum}
\mathbf{EA}(\mathbf{L}(\lambda))\cong
V'(0,\tilde{\lambda})\oplus V'(2,\tilde{\lambda}).
\end{equation}

Consider now the module $L(1)\otimes V'(0,\tilde{\lambda})$. It has
$\g$-types $L(1)$, $L(3)$, $L(5),\dots$, all multiplicity-free.
We claim that that $L(1)\otimes V'(0,\tilde{\lambda})$ is a simple
$\g$-module which we can declare to be $V(1,\tilde{\lambda})$.
Assume that $L(1)\otimes V'(0,\tilde{\lambda})$ is not simple and let $V$
be a submodule or a quotient of $L(1)\otimes V'(0,\tilde{\lambda})$ which does not have
$L(1)$ as its $\g$-type. Using the self-adjointness of 
$L(1)\otimes -$, by adjunction we have a non-zero homomorphism between
$V'(0,\tilde{\lambda})$ and $L(1)\otimes V$. However, the latter is
not possible as $L(0)$ is not a $\g$-type of $L(1)\otimes V$ due to
our definition of $V$. This shows that $L(1)\otimes V'(0,\tilde{\lambda})$ is simple.

From Lemma~\ref{lemnew1403985y}, it follows that 
$L(1)\otimes \mathbf{L}(\lambda)\cong \mathbf{L}(\lambda-\frac{1}{2}\alpha)\oplus 
\mathbf{L}(\lambda+\frac{1}{2}\alpha)$.
As $\mathbf{EA}$ commutes with $L(1)\otimes -$, it follows that
\begin{displaymath}
L(1)\otimes \mathbf{EA}(\mathbf{L}(\lambda))\cong
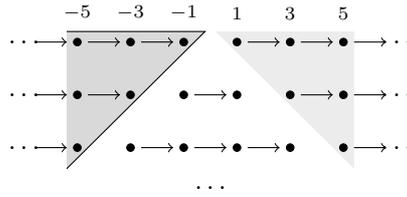
\mathbf{EA}(\mathbf{L}(\lambda-\frac{1}{2}\alpha))\oplus 
\mathbf{EA}(\mathbf{L}(\lambda+\frac{1}{2}\alpha)).
\end{displaymath}
By comparing the $\g$-types, we see that 
\begin{displaymath}
\mathbf{EA}(\mathbf{L}(\lambda-\frac{1}{2}\alpha))\cong V(1,\tilde{\lambda})\quad\text{ or }\quad
\mathbf{EA}(\mathbf{L}(\lambda+\frac{1}{2}\alpha))\cong V(1,\tilde{\lambda}).
\end{displaymath}
Consider the first case, the second one is similar. By adjunction, we have
\begin{displaymath}
\mathrm{Hom}(V(1,\tilde{\lambda}),L(1)\otimes \mathbf{EA}(\mathbf{L}(\lambda)))\cong
\mathrm{Hom}(L(1)\otimes V(1,\tilde{\lambda}),\mathbf{EA}(\mathbf{L}(\lambda))).
\end{displaymath}
By the above, $L(1)\otimes V(1,\tilde{\lambda})$ has 
$\mathbf{EA}(\mathbf{L}(\lambda))$ as a direct summand and the endomorphism of
$\mathbf{EA}(\mathbf{L}(\lambda))$ has dimension two. Consequently,
$\mathrm{Hom}(V(1,\tilde{\lambda}),\mathbf{L}(\lambda+\frac{1}{2}\alpha))$
must be non-zero, which implies 
\begin{displaymath}
\EnAr (\mathbf{L}(\lambda+\frac{1}{2}\alpha)) \cong V(1,\tilde{\lambda})
\end{displaymath}
by comparing the $\g$-types of these modules.

Using, inductively, arguments similar to the ones used above, we construct modules
\begin{displaymath}
V(i,\tilde{\lambda})\cong \EnAr(\mathbf{L}(\lambda+\frac{i}{2}\alpha))
\end{displaymath}
for $i>1$. The proof of  Theorem~\ref{thmconf4new} is now completed easily using
construction and adjunction.
\end{proof}

\begin{remark}[A sketch of an altenative proof of the splitting (\ref{equation:EA=sum})]
Denote by $\bb$ the span of $h$ and $e$, and consider the polynomials $\C[x]$ as $(\bb \ltimes L(4))$-module by declaring: $h \cdot x^k = (-2-2k)x^k$, $e$ acts as $\frac{\partial}{\partial x}$, $v_4$ and $v_2$ annihilate everything, $v_0$ multiplies by $3 \tilde{\lambda}$, $v_{-2}$ multiplies by $6 \tilde{\lambda} x$, and $v_{-2}$ multiplies by $3 \tilde{\lambda} x^2$ (this is known as the Fock module). From \cite[Theorem 4(ii)]{lu2014simple} it follows that $\mathbf{L}(\lambda) \cong \Ind_{\bb \ltimes L(4)}^{\LL^4} \C[x]$ (recall that $\lambda(h)=-2$ and $\lambda(v_0)=\tilde{\lambda}$). From this, we have that $\EnAr(\mathbf{L}(\lambda))$ has $\g$-types $L(0), L(2), L(4), \dots$, and the lowest weight vector in each $L(2k)$ is $f^{-1} \otimes x^k$. By a long and tedious, but straightforward computation, one can see that
\begin{align*}
v_4 f^{-1} \otimes x^k = \frac{3 \tilde{\lambda}}{4(k+1)(k+2)(2k+1)(2k+3)} {}\cdot{} & e^4 f^{-1} \otimes x^{k+2} \\
- \frac{3 \tilde{\lambda}}{2(2k+3)(2k-1)} {}\cdot{} & e^2 f^{-1} \otimes x^k \\
+ \frac{3 \tilde{\lambda} k (k-1)}{4(2k+1)(2k-1)} {}\cdot{} & f^{-1} \otimes x^{k-2} .\end{align*}
From this formula it follows easily that $\LL^4$ maps $L(k)$ to $L(k+4) \oplus L(k) \oplus L(k-4)$ if $k \geq 2$, and to $L(k+2) \oplus L(k)$ for $k=0,1$, with non-zero projections to each summand. This proves the claim.
\end{remark}

\begin{remark}\label{remark:comb_struct_conf4}
From the proof of Theorem~\ref{thmconf4new},
the combinatorics of tensoring with $L(1)$ can be recorded as follows:
\begin{displaymath}
\begin{array}{ccll}
L(1)\otimes  V'(0,\tilde{\lambda})&\cong& V(1,\tilde{\lambda});\\
L(1)\otimes  V'(2,\tilde{\lambda})&\cong& V(1,\tilde{\lambda});\\
L(1)\otimes  V(1,\tilde{\lambda})&\cong& V'(0,\tilde{\lambda})\oplus 
V'(2,\tilde{\lambda})\oplus V(2,\tilde{\lambda});\\
L(1)\otimes  V(i,\tilde{\lambda})&\cong& V(i-1,\tilde{\lambda})\oplus V(i+1,\tilde{\lambda}),
&i>1.\\
\end{array}
\end{displaymath}
In particular, the additive closure of $V'(0,\tilde{\lambda})$,
$V'(2,\tilde{\lambda})$ and all $V(i,\tilde{\lambda})$ forms a
simple $\cF$-module category, cf. Propositions~\ref{propact1} and \ref{propact2}.
From the above formulae, we see that the combinatorics of this 
$\cF$-module category is different from the combinatorics
of the $\cF$-module categories described in 
Propositions~\ref{propact1} and \ref{propact2}.
\end{remark}

We note that there might exist, potentially, other simple
$\g$-Harish-Chandra modules for $\LL^4$ which correspond to 
characters of $\Sym(\rr)^\g$ that do not occur in simple 
highest weight modules. The problem here seems to be the
absence, in case of $\LL^4$, of an analogue of the classical 
theorem by Harish-Chandra that, in case of reductive Lie algebras,
any central character is realizable on some highest weight module.

\subsection{Some speculations in case of the Lie algebras $\LL^n = \mathfrak{sl}_2\ltimes L(n)$ with $n$ even}

For an even non-negative integer $n$, consider the Lie algebra 
$\LL^n = \mathfrak{sl}_2 \ltimes L(n)$. The algebra $\LL^0$ is reductive
and hence all its $\mathfrak{sl}_2$-Harish-Chandra modules are 
finite-dimensional. Classification of all $\mathfrak{sl}_2$-Harish-Chandra modules
for the Takiff Lie algebra $\LL^2$ is given in Section~\ref{section:takiff_sl2}.
Note that all simple $\mathfrak{sl}_2$-Harish-Chandra modules for $\LL^2$
are connected to highest weight $\LL^2$-modules. For the Lie algebra $\LL^4$,
all $\mathfrak{sl}_2$-Harish-Chandra modules that are connected to 
highest weight $\LL^4$-modules are classified in the previous subsection.
Potentially, these are not all $\mathfrak{sl}_2$-Harish-Chandra modules for $\LL^4$.
As we see, the level of difficulty of the problem increases drastically with $n$.

The highest weight theory for $\LL^n$ is described in \cite[Theorem~4]{lu2014simple}.
As vector space, simple highest weight $\LL^n$-modules look the same, independently
of $n$. Therefore we expect that the problem of classification of simple 
$\mathfrak{sl}_2$-Harish-Chandra modules for $\LL^n$ that are connected to 
highest weight $\LL^n$-modules via Enright-Arkhipov functor should be solvable.
One of the crucial missing ingredients, at the moment, seems to be an analogue of 
\cite[Corollary~3.4]{hahn2018from} for general $n$ (i.e., for general $k$ in the notation of \cite{hahn2018from}).

Of special interest is the question of what kind of a monoidal representation of
the monoidal category of finite-dimensional $\mathfrak{sl}_2$-modules do the
$\LL^n$-modules from the previous paragraph form.

\noindent
V.~M.: Department of Mathematics, Uppsala University, Box. 480,
SE-75106, Uppsala, SWEDEN, email: {\tt mazor\symbol{64}math.uu.se}

\noindent
R.~M.: Department of Mathematics, Uppsala University, Box. 480,
SE-75106, Uppsala, SWEDEN, email: {\tt rafael.mrden\symbol{64}math.uu.se}\\
On leave from: Faculty of Civil Engineering, University of Zagreb, Fra 
Andrije Ka\v{c}i\'{c}a-Mio\v{s}i\'{c}a 26, 10000 Zagreb, CROATIA


\begin{thebibliography}{DLMZ14}

\bibitem[AIM19]{alshammari2019on}
F.~Alshammari, P.~S. Isaac, and I.~Marquette.
\newblock On {C}asimir operators of conformal {G}alilei algebras.
\newblock {\em J. Math. Phys.}, 60(1):013509, 14, 2019.

\bibitem[Ark04]{arkhipov2004algebraic}
S.~Arkhipov.
\newblock Algebraic construction of contragradient quasi-{V}erma modules in
  positive characteristic.
\newblock In {\em Representation theory of algebraic groups and quantum
  groups}, volume~40 of {\em Adv. Stud. Pure Math.}, pages 27--68. Math. Soc.
  Japan, Tokyo, 2004.

\bibitem[AS03]{andersen2003twisting}
H.~H. Andersen and C.~Stroppel.
\newblock Twisting functors on {$\mathcal{O}$}.
\newblock {\em Represent. Theory}, 7:681--699, 2003.

\bibitem[BM]{BM}
P.~Batra, V.~Mazorchuk.
\newblock Blocks and modules for Whittaker pairs. 
\newblock {\em J. Pure Appl. Algebra}, 215(7):1552--1568, 2011.

\bibitem[BG09]{bagchi2009galilean}
A.~Bagchi and R.~Gopakumar.
\newblock Galilean conformal algebras and {AdS}/{CFT}.
\newblock {\em J. High Energy Phys.}, 2009(07):037--037, 2009.

\bibitem[BL17]{bavula2017prime}
V.~V. Bavula and T.~Lu.
\newblock Prime ideals of the enveloping algebra of the {E}uclidean algebra and
  a classification of its simple weight modules.
\newblock {\em J. Math. Phys.}, 58(1):011701, 33, 2017.

\bibitem[BL18]{bavula2018classification}
V.~V. Bavula and T.~Lu.
\newblock Classification of simple weight modules over the {S}chr\"{o}dinger
  algebra.
\newblock {\em Canad. Math. Bull.}, 61(1):16--39, 2018.

\bibitem[Blo]{Bl}
R.~E.~Block. 
\newblock The irreducible representations of the Lie algebra  $sl(2)$ and of the Weyl algebra. 
\newblock {\em Adv. in Math.}, 39(1):69--110, 1981.


\bibitem[Bou89]{bourbaki1989lie}
N.~Bourbaki.
\newblock {\em Lie groups and {L}ie algebras. {C}hapters 1--3}.
\newblock Elements of Mathematics (Berlin). Springer-Verlag, Berlin, 1989.
\newblock Translated from the French, Reprint of the 1975 edition.

\bibitem[CC17]{cai2017quasi}
Y.~Cai and Q.~Chen.
\newblock Quasi-{W}hittaker modules over the conformal {G}alilei algebras.
\newblock {\em Linear Multilinear Algebra}, 65(2):313--324, 2017.

\bibitem[CCS14]{cai2014quasi}
Y.~Cai, Y.~Cheng, and R.~Shen.
\newblock Quasi-{W}hittaker modules for the {S}chr\"{o}dinger algebra.
\newblock {\em Linear Algebra Appl.}, 463:16--32, 2014.

\bibitem[CSZ16]{cai2016whittaker}
Y.~Cai, R.~Shen, and J.~Zhang.
\newblock {W}hittaker modules and quasi-{W}hittaker modules for the {E}uclidean
  {L}ie algebra $\mathfrak{e}(3)$.
\newblock {\em J. Pure Appl. Algebra}, 220(4):1419--1433, 2016.

\bibitem[CCM]{CCM}
C.-W.~Chen, K.~Coulembier, V.~Mazorchuk.
\newblock Translated simple modules for Lie algebras and simple supermodules for Lie superalgebras.
\newblock Preprint arXiv:1807.03834.

\bibitem[CM]{CM}
C.-W.~Chen, V.~Mazorchuk.
\newblock Simple supermodules over Lie superalgebras.
\newblock Preprint arXiv:1801.00654.

\bibitem[Deo80]{deodhar1980on}
V.~V. Deodhar.
\newblock On a construction of representations and a problem of {E}nright.
\newblock {\em Invent. Math.}, 57(2):101--118, 1980.

\bibitem[DMP]{DMP}
I.~Dimitrov, O.~Mathieu, I.~Penkov.
\newblock On the structure of weight modules. 
\newblock {\em Trans. Amer. Math. Soc.}, 352(6):2857--2869, 2000. 

\bibitem[DDM97]{dobrev1997lowest}
V.~K. Dobrev, H.-D. Doebner, and C.~Mrugalla.
\newblock Lowest weight representations of the {S}chr\"{o}dinger algebra and
  generalized heat {S}chr\"{o}dinger equations.
\newblock {\em Rep. Math. Phys.}, 39(2):201--218, 1997.

\bibitem[DOF]{DOF}
Yu.~A.~Drozd, S.~A.~Ovsienko, V.~M.~Futorny. 
\newblock On Gelfand-Zetlin modules. 
\newblock Proceedings of the Winter School on Geometry and Physics (Srn{\'\i}, 1990). 
Rend. Circ. Mat. Palermo (2) Suppl. No. 26, 143--147, 1991


\bibitem[DLMZ14]{dubsky2014category}
B.~Dubsky, R.~L\"{u}, V.~Mazorchuk, and K.~Zhao.
\newblock Category {$\mathcal{O}$} for the {S}chr\"{o}dinger algebra.
\newblock {\em Linear Algebra Appl.}, 460:17--50, 2014.

\bibitem[Dub14]{dubsky2014classification}
B.~Dubsky.
\newblock Classification of simple weight modules with finite-dimensional
  weight spaces over the {S}chr\"{o}dinger algebra.
\newblock {\em Linear Algebra Appl.}, 443:204--214, 2014.

\bibitem[EMV]{EMV}
N.~Early, V.~Mazorchuk, E.~Vishnyakova.
\newblock Canonical Gelfand-Zeitlin modules over orthogonal Gelfand-Zeitlin algebras.
\newblock Preprint arXiv:1709.01553. To appear in IMRN.


\bibitem[Enr79]{enright1979on}
T.~J. Enright.
\newblock On the fundamental series of a real semi-simple {L}ie algebra: their
  irreducibility, resolutions and multiplicity formulae.
\newblock {\em Ann. of Math. (2)}, 110(1):1--82, 1979.

\bibitem[FO05]{futorny2005kostant}
V.~Futorny and S.~Ovsienko.
\newblock Kostant's theorem for special filtered algebras.
\newblock {\em Bull. London Math. Soc.}, 37(2):187--199, 2005.

\bibitem[Geo94]{geoffriau1994centre}
F.~Geoffriau.
\newblock Sur le centre de l'alg\`ebre enveloppante d'une alg\`ebre de
  {T}akiff.
\newblock {\em Ann. Math. Blaise Pascal}, 1(2):15--31 (1995), 1994.

\bibitem[Geo95]{geoffriau1995homomorphisme}
F.~Geoffriau.
\newblock Homomorphisme de {H}arish-{C}handra pour les alg\`ebres de {T}akiff
  g\'{e}n\'{e}ralis\'{e}es.
\newblock {\em J. Algebra}, 171(2):444--456, 1995.

\bibitem[GK12]{gomis2012schrodinger}
J.~Gomis and K.~Kamimura.
\newblock Schr\"{o}dinger equations for higher order nonrelativistic particles
  and $n$-{G}alilean conformal symmetry.
\newblock {\em Phys. Rev. D}, 85:045023, Feb 2012.

\bibitem[HHLS18]{hahn2018from}
H.~Hahn, J.~Huh, E.~Lim, and J.~Sohn.
\newblock From partition identities to a combinatorial approach to explicit
  {S}atake inversion.
\newblock {\em Ann. Comb.}, 22(3):543--562, 2018.

\bibitem[Han19]{han2019higher}
B.~Han.
\newblock {\em Higher Spin Algebras and Universal Enveloping Algebras}.
\newblock Australian National University, 2019.
\newblock Bachelor thesis.

\bibitem[Hil93]{hilbert1993theory}
D.~Hilbert.
\newblock {\em Theory of algebraic invariants}.
\newblock Cambridge University Press, Cambridge, 1993.
\newblock Translated from the German and with a preface by Reinhard C.
  Laubenbacher, Edited and with an introduction by Bernd Sturmfels.

\bibitem[Hum72]{humphreys1978introduction}
J.~E. Humphreys.
\newblock {\em Introduction to {L}ie algebras and representation theory}.
\newblock Springer-Verlag, New York-Berlin, 1972.
\newblock Graduate Texts in Mathematics, Vol. 9.

\bibitem[Hum08]{humphreys2008representations}
J.~E. Humphreys.
\newblock {\em Representations of semi-simple {L}ie algebras in the {BGG}
  category {$\mathcal{O}$}}, volume~94 of {\em Graduate Studies in
  Mathematics}.
\newblock American Mathematical Society, Providence, RI, 2008.

\bibitem[Jan83]{jantzen11983einhullende}
J.~C. Jantzen.
\newblock {\em Einh\"{u}llende {A}lgebren halbeinfacher {L}ie-{A}lgebren},
  volume~3 of {\em Ergebnisse der Mathematik und ihrer Grenzgebiete (3)
  [Results in Mathematics and Related Areas (3)]}.
\newblock Springer-Verlag, Berlin, 1983.

\bibitem[Kac90]{kac1990infinite}
V.~G. Kac.
\newblock {\em Infinite-dimensional {L}ie algebras}.
\newblock Cambridge University Press, Cambridge, third edition, 1990.

\bibitem[KM02]{konig2002enrights}
S.~K\"{o}nig and V.~Mazorchuk.
\newblock {E}nright's completions and injectively copresented modules.
\newblock {\em Trans. Amer. Math. Soc.}, 354(7):2725--2743, 2002.

\bibitem[KM05]{khomenko2005on}
O.~Khomenko and V.~Mazorchuk.
\newblock On {A}rkhipov's and {E}nright's functors.
\newblock {\em Math. Z.}, 249(2):357--386, 2005.

\bibitem[Kna88]{knapp1988lie}
A.~W. Knapp.
\newblock {\em Lie groups, {L}ie algebras, and cohomology}, volume~34 of {\em
  Mathematical Notes}.
\newblock Princeton University Press, Princeton, NJ, 1988.

\bibitem[Kos78]{kostant1978on}
B.~Kostant.
\newblock On {W}hittaker vectors and representation theory.
\newblock {\em Invent. Math.}, 48(2):101--184, 1978.

\bibitem[KL00]{krause2000growth}
G.~R. Krause and T.~H. Lenagan.
\newblock {\em Growth of algebras and {G}elfand-{K}irillov dimension},
  volume~22 of {\em Graduate Studies in Mathematics}.
\newblock American Mathematical Society, Providence, RI, revised edition, 2000.

\bibitem[Lau18]{lau2018classification}
M.~Lau.
\newblock Classification of {H}arish-{C}handra modules for current algebras.
\newblock {\em Proc. Amer. Math. Soc.}, 146(3):1015--1029, 2018.

\bibitem[LMZ14]{lu2014simple}
R.~L\"{u}, V.~Mazorchuk, and K.~Zhao.
\newblock On simple modules over conformal {G}alilei algebras.
\newblock {\em J. Pure Appl. Algebra}, 218(10):1885--1899, 2014.

\bibitem[Mat00]{mathieu2000classification}
O.~Mathieu.
\newblock Classification of irreducible weight modules.
\newblock {\em Ann. Inst. Fourier (Grenoble)}, 50(2):537--592, 2000.

\bibitem[Maz10]{mazorchuk2010lectures}
V.~Mazorchuk.
\newblock {\em Lectures on {$\mathfrak{sl}_2(\mathbb{C})$}-modules}.
\newblock Imperial College Press, London, 2010.

\bibitem[MS19]{mazorchuk2019category}
V.~Mazorchuk and C.~S\"{o}derberg.
\newblock Category $\mathcal{O}$ for {T}akiff $\mathfrak{sl}_2$.
\newblock {\em J. Math. Phys.}, 60(11):111702, 2019.

\bibitem[MSt07]{MS}
V.~Mazorchuk, C.~Stroppel. 
\newblock On functors associated to a simple root. 
\newblock {\em J. Algebra}, 314(1):97--128, 2007. 

\bibitem[Mol]{molev1996casimir}
A.~Molev.
\newblock Casimir elements for certain polynomial current lie algebras.
\newblock In H.-D. Doebner, P.~Nattermann, and W.~Scherer, editors, {\em Group
  21: Physical Applications and Mathematical Aspects of Geometry, Groups, and
  Algebras}, pages 172--176. World Scientific.


\bibitem[Per77]{perroud1977projective}
M.~Perroud.
\newblock Projective representations of the {S}chr\"{o}dinger group.
\newblock {\em Helv. Phys. Acta}, 50(2):233--252, 1977.

\bibitem[PRS90]{pope1990W_infinity}
C.~N. Pope, L.~J. Romans, and X.~Shen.
\newblock {$W_\infty$ and the {Racah-Wigner} algebra}.
\newblock {\em Nucl. Phys.}, B339:191--221, 1990.

\bibitem[Tak71]{takiff1971rings}
S.~J. Takiff.
\newblock Rings of invariant polynomials for a class of {L}ie algebras.
\newblock {\em Trans. Amer. Math. Soc.}, 160:249--262, 1971.

\bibitem[Ve]{Ve}
D.-N.~Verma.
\newblock Structure of certain induced representations of complex semi-simple Lie algebras. 
\newblock {\em Bull. Amer. Math. Soc.}, 74:160--166, 1968.

\bibitem[Vog81]{vogan1981repersentations}
D.~A. Vogan.
\newblock {\em Representations of real reductive {L}ie groups}, volume~15 of
  {\em Progress in Mathematics}.
\newblock Birkh\"{a}user, Boston, Mass., 1981.


\bibitem[Web]{We}
B.~Webster.
\newblock Gelfand-Tsetlin modules in the Coulomb context.
\newblock Preprint arXiv:1904.05415. 

\bibitem[Wil11]{wilson2011highest}
B.~J. Wilson.
\newblock Highest-weight theory for truncated current {L}ie algebras.
\newblock {\em J. Algebra}, 336:1--27, 2011.

\end{thebibliography}
\end{document}